\documentclass[12pt]{article}
\usepackage[margin=1.1in]{geometry}
\usepackage[english]{babel}
\usepackage[numbers]{natbib}
\usepackage{amsmath}
\usepackage{amssymb}
\usepackage{amsthm}
\usepackage{amscd}
\usepackage{bbm}
\usepackage{enumerate}
\usepackage{esvect}
\usepackage[inline]{asymptote}
\usepackage{hyperref}
\usepackage{cleveref}
\usepackage{pgf,tikz}

\usetikzlibrary{arrows}
\usetikzlibrary{arrows.meta}

\newtheorem{theorem}{Theorem}[section]
\newtheorem{lemma}[theorem]{Lemma}

\newtheorem{corollary}[theorem]{Corollary}

\theoremstyle{definition}
\newtheorem{definition}[theorem]{Definition}

\theoremstyle{remark}
\newtheorem{remark}[theorem]{Remark}

\newcommand{\pref}[1]{(\ref{#1})}

\newcommand{\product}[1]{\left\langle#1\right\rangle}

\newcommand{\torus}{\mathbb{T}}
\newcommand{\p}{\mathfrak{p}}

\newcommand{\cM}{\mathcal{M}}

\newcommand{\cX}{\mathcal{X}}
\newcommand{\cF}{\mathcal{F}}

\allowdisplaybreaks

\title{Synchronization of mean-field models on the circle}
\author{Y. Polyanskiy\thanks{Department of EECS, MIT. Email: yp@mit.edu}, P. Rigollet\thanks{Department of Mathematics, MIT. Email: rigollet@math.mit.edu}, A. Yao\thanks{Departments of EECS and Mathematics, MIT. Email: ajyao@mit.edu}}

\date{January 2025}

\begin{document}

\maketitle

\begin{abstract} This paper considers a mean-field model of $n$ interacting particles whose state
space is the unit circle, a generalization of the classical Kuramoto model. Global synchronization is said to occur if after starting from almost any
initial state, all particles coalesce to a common point on the circle. 
We propose a general synchronization criterion in terms of $L_1$-norm of the third derivative of the particle
interaction function. As an application we resolve a
conjecture for the so-called self-attention dynamics (stylized model of transformers), by showing synchronization for all $\beta \ge
-0.16$, which significantly extends the previous bound of $0\le \beta \le 1$ from
Criscitiello, Rebjock, McRae, and Boumal~\cite{synchronization_nonlinear}. We also show that
global synchronization does not occur when $\beta < -2/3$.
\end{abstract}

\section{Introduction}

In this paper, we consider a simple dynamical system consisting of $n$ interacting particles $x_1,\ldots,x_n$ situated on the unit sphere $\mathbb{S}^1$, which we identify with the standard 1-torus $\mathbb{S}^1 \simeq \torus \triangleq \mathbb{R}/2\pi\mathbb{Z}$. The dynamics are given by
\begin{equation}\tag{S1}
\label{eq:generalsystem2}
\dot{x}_i(t) = -\sum_{j=1}^n f(x_i(t)-x_j(t)), \forall i\in[n].
\end{equation}

Interest in such systems originated with the work of Kuramoto \cite{kuramoto}, who analyzed
the special case where $f(x) = \sin(x)$.\footnote{We consider only the so-called
\textit{homogeneous} Kuramoto model, without particle-dependent drift terms.} Kuramoto discovered 
a fascinating synchronization phenomenon in this system, which we define below.

\begin{definition}
\label{def:synchronization}
In \pref{eq:generalsystem2}, \textit{synchronization} occurs for a starting point $(x_i(0))_{1\leq i\leq n}$ if there exists $x^*\in [0, 2\pi)$ such that $\lim_{t\rightarrow\infty} x_i(t) \equiv x^*$ for all $i\in [n]$. We say that a pair $(A,f)$ exhibits \textit{global synchronization} if synchronization occurs for almost every starting point (with respect to the volume measure).
\end{definition}

Mathematical models of synchronization, such as Kuramoto's, are simplified abstractions of the
ubiquitous self-organization phenomena observed in nature~\cite{winfree1967biological}. While many other
models have been proposed in physics~\cite{vicsek1995novel}, biology~\cite{mirollo1990synchronization}, and
engineering~\cite{blondel2005convergence,jadbabaie2003coordination}, Kuramoto's model remains the simplest to exhibit
this effect.

The seminal paper of Kuramoto~\cite{kuramoto} has ushered a rich line of work sharpening, generalizing, and applying his original framework; see~\cite{strogatz2000kuramoto,acebron2005kuramoto}. After successive advances, global synchronization for the mean-field Kuramoto model was eventually established in \cite{Taylor2012}, which also expanded the original Kuramoto model beyond the mean-field case by allowing each particle to interact with only a subset of the others; see~\cite{expander,randomgraph} for the most recent advances on this front.

The dynamical system~\eqref{eq:generalsystem2} is an example of a \textit{mean-field} model, since
its motion can be rewritten as:
$$ \dot x_i(t) = \cX[\mu_n(t)](x_i)\,,\qquad \mu_n(t) \triangleq {1\over n} \sum_{j=1}^n
\delta_{x_j(t)}\,, \quad \cX[\mu](x) \triangleq -\int_\torus f(x-y) \mu(dy)\,,$$
where $\mu_n(t)$ denotes the empirical measure (distribution) of the particles and $\cX[\mu]$ is
the measure-dependent vector field driving each particle.
Since particles are indistinguishable in mean-field models, it suffices to study the evolution of
the empirical measure $\mu_n(t)$, which satisfies the
continuity equation:
\begin{equation}\label{eq:ctty}
	\dot \mu(t) + \mathrm{div}(\mu \cX[\mu]) = 0\,. 
\end{equation}
When $\mu(0) = \mu_n(0) = {1\over n} \sum_{j=1}^n \delta_{x_j(0)}$, studying~\eqref{eq:ctty} is
equivalent to studying~\eqref{eq:generalsystem2}, but one can also study solutions
of~\eqref{eq:ctty} starting from non-discrete measures $\mu(0)$.

For the Kuramoto mean-field model, i.e. $f(x)=\sin(x)$,~\cite{frouvelle2017long} 
showed explicit (exponential) estimates of speed of
convergence to synchronization for the dynamical system~\eqref{eq:generalsystem2} and,
subsequently,~\cite{morales2022trend} extend the exponential convergence to the more general case
of evolution of measures solving continuity equation~\eqref{eq:ctty}. Both works in fact
consider a more general case of the Kuramoto mean-field model with state space $\mathbb{S}^d$ with $d\ge 1$. 

A recent resurgence of interest in mean-field models on the sphere and torus arose from a
discovery of~\cite{mathpersp25} that with $f(x) = f_\beta(x) \triangleq \sin(x)e^{\beta \cos (x)},
\beta \in \mathbb{R}$ the resulting interacting particle system is intimately related to evolution
of internal representations in transformers~\cite{transformers}, which are modern neural networks
forming the backbone of large language models (LLMs). When $f=f_\beta$, we
call~\eqref{eq:generalsystem2} \textit{self-attention dynamics}, on which there is a fast
growing body of
work~\cite{sander2022sinkformers,clusters_selfattention,clustering_causal,metastability,synchronization_nonlinear,metastability,bruno2024emergence,abella2024asymptotic,burger2025analysis,castin2025unified,chen2025quantitative}.
Despite the complexity of practical transformers, the simple model of self-attention dynamics is
remarkably effective at predicting how signals propagate through internal layers. From the
practical point of view, global synchronization is an abstraction of a
complex phenomenon in LLMs known as clustering or oversmoothing,
e.g.~\cite{feng2022rank,shi2022revisiting,dovonon2024setting,koubbi2024impact}.

The work~\cite{mathpersp25} establishes global synchronization whenever
$\beta =O(1/n)$ or $\beta = \Omega(n^2)$ (in both cases $\beta\ge0$ is also required) and for all state spaces $\mathbb{S}^d$,
$d\ge 1$. Shortly afterwards,~\cite{synchronization_nonlinear} made an important observation that an
earlier work of~\cite{markdahl} in fact shows global synchronization for all $\beta\ge 0$ and
$d\ge 2$. For $d=1$, the authors of \cite{synchronization_nonlinear} improved the argument
of~\cite{mathpersp25} and showed synchronization for $\beta\le 1$. In~\cite{clustering_causal}, the
results on global synchronization are further generalized but under the additional assumption that
the summation defining $\dot x_i$ in~\eqref{eq:generalsystem2} only extends to $j<i$, which
corresponds to a simple model of the ubiquitous
``decoder-only'' LLMs implementing next-token prediction.

\textbf{Meta-stability.} 
The novel aspect of self-attention dynamics compared to Kuramoto's
model is the emergence of meta-stability above a certain critical value of $\beta>0$. Specifically, ~\cite{mathpersp25} observed empirically
that once $\beta$ is sufficiently large, then particles initialized
i.i.d. uniformly evolve in two phases: 
first, they quickly group into $\asymp\sqrt{\beta}$ tight clusters and then over a much slower
time-horizon the clusters progressively merge until only one remains, thus attaining global
synchronization. In~\cite{metastability} it is confirmed that localized groups of
particles contract exponentially quickly to their common center. Bruno, Pasqualotto, and Agazzi in \cite{bruno2024emergence} showed that after initializing the particles
$x_i(0)$ i.i.d. from the uniform measure on the circle, at time $t\asymp \log n$ the empirical measure
$\mu_n(t)$ develops periodic lumps with high probability. More explicitly, they show that for any $0<\delta \ll 1$, there exists a ${2\pi\over k}$-periodic
probability distribution $\nu_{per}(t)$, with $k\asymp\sqrt{\beta}$,  which is $\delta$-away from uniform (as measured, for
example, by the Wasserstein
distance $W_1$) and $W_1(\mu_n(t), \nu_{per}(t)) \to 0$ in probability as $n \to \infty$ for 
$t = t(n,\delta,\beta) \asymp \log n$.

In the present work, we show that with probability 1, for any fixed $n$ and $\beta\ge 0$, we must have
that $\mu_n(t) \to \delta_{x_\infty}$ as $t \to \infty$. In turn, this implies that although the ${2\pi\over k}$-periodic
phase is rather long-lived, it will eventually collapse, which implies that it is meta-stable. 

\textbf{Contributions.} 
In addition to~\eqref{eq:generalsystem2}, in the context of Transformers a so-called ``normalized'' version of this dynamics is also
important. This generalization of~\eqref{eq:generalsystem2} can be stated in the following form:
\begin{equation}\tag{S2}
\label{eq:systemnormalized}
\dot{x}_i(t) = -\frac{1}{g_i(x_1(t),\ldots,x_n(t))} \sum_{j=1}^n f(x_i(t)-x_j(t)), \, 1\leq i\leq n,
\end{equation}
where $g: \mathbb{T}^n\rightarrow \mathbb{R}_{>0}^n$ is some smooth function.
In this work, we propose a general criterion for systems~\eqref{eq:generalsystem2}
and~\eqref{eq:systemnormalized} with state-space
$\mathbb{S}^1=\mathbb{T}$ to be globally synchronizing. As an application, we prove (a) global synchronization for transformers
(i.e. $f=f_\beta$) for all $\beta \ge 0$, thus completing the study of this class of interaction functions; (b) initiate
the study of $\beta < 0$ and show global synchronization for $\beta \geq -0.16$ and
non-synchronization for $\beta < -2/3$. Finally, in Section~\ref{sec:generalizedsystem} we extend
the criterion to a certain class of non-mean-field systems.

\textbf{Organization.} Section~\ref{sec:main} states all of our results formally. Section~\ref{sec:mainproof} contains proof of the main criterion for stability of system~\eqref{eq:generalsystem2}, i.e. Theorem~\ref{thm:main}. Section~\ref{sec:normalized} extends the results to the normalized system~\eqref{eq:systemnormalized}. Section~\ref{sec:application} verifies that the general criterion in Theorem~\ref{thm:main} applies to Transformer dynamics on the circle ($f=f_\beta$ for all $\beta > -0.16$). Finally, in Section~\ref{sec:generalizedsystem} we further generalize the results to the dynamics where particles are aggregated with unequal weights.

\textbf{Acknowledgments.} PR is supported by NSF grants DMS-2022448 and \\CCF-2106377.

\section{Main results}\label{sec:main}

Given a smooth vector field $\cF:\cM \to T\cM$ on a manifold $\cM$, a dynamical system solves the ordinary differential equation (ODE): 
    $$ \dot {\mathbf{x}} = \cF(\mathbf{x})\,. $$
Point $\mathbf{x}$ is called stationary if $\cF(\mathbf{x})=0$, since started from this point the system does not move: $\dot {\mathbf{x}} = 0$. If the trajectory of a dynamical system converges, then the limiting point should necessarily be a stationary point. We call a stationary point $\mathbf{x}$ \textit{locally unstable} if the Jacobian of $\cF$ at $\mathbf{x}$ has an eigenvalue with positive real part. Note that this implies that for a small neighborhood $U$ around $\mathbf{x}$, almost all initializations $\mathbf{x}_0 \in U$ result in trajectories that escape from $U$.

The dynamical system~\eqref{eq:generalsystem2} that we consider here corresponds to taking $\cM = \torus^n$ and the vector field with $i$-th
component being
$$ \cF(\mathbf{x})_i = \mathcal{X}[\mu_n](x_i) = - \int f(x_i - y) d\mu_n(y)\,.$$
As we discussed, mean-field models can also be thought of as $n$ exchangeable particles (each with
state space of $\torus$) each driven by a time-dependent vector field $\cX[\mu_n(t)]: \torus \to
T\torus$, which is a function of the empirical measure $\mu_n$, cf.~\eqref{eq:ctty}.

\begin{theorem}
\label{thm:main} Consider the mean-field model~\eqref{eq:generalsystem2} on $\torus$. Let $\tau\in(0,\pi]$ satisfy $f'(x) < 0$ for all $x\not\in[-\tau,\tau]$. If 
\[
\tau\int_{-\pi}^\pi |f'''(x)|_+ dx\leq 4\left(1+\frac{\tau}{2\pi}\right)f'(0),
\]
then every stationary point $(x_1, \ldots, x_n)$  of the system~\eqref{eq:generalsystem2} on $\mathbb{T}^n$ is either locally unstable or synchronized (i.e. $x_1=\cdots=x_n$).
\end{theorem}

This characterization constitutes the main result of our work. However, to establish global synchronization in systems~\eqref{eq:generalsystem2} and~\eqref{eq:systemnormalized}, we require two additional (though now standard) ingredients: {\L}ojasiewicz's theorem and the center-stable manifold theorem.

\subsection{Gradient ascent dynamics}

The first obstacle to synchronization could be the emergence of limit cycles. It turns out,
however, that Transformer dynamics is special since it can be  written as a gradient ascent for a
certain energy function $\mathsf{E}(\mathbf{x})$, see~\cite[(3.5)]{mathpersp25}
and~\eqref{eq:energy_sad} below. Consequently, as explained in \cite[Appendix A]{mathpersp25} (also \cite[Corollary 5.1]{lohe}), classical {\L}ojasiewicz's theorem \cite{lojasiewicz} guarantees that as long as $\mathsf{E}$ is real analytic, the gradient ascent dynamics $\dot {\mathbf{x}} = \nabla_\cM \mathsf{E}(\mathbf{x})$ over a  compact Riemannian manifold must converge to some stationary point $\mathbf{x}_\infty$. (We denote $\nabla_\cM$ the Riemannian gradient on a manifold, see \cite{intromanifolds} for details.)

The next issue is that even for $f=f_\beta$, the system~\eqref{eq:generalsystem2} has  many stationary points other than the synchronized ones (for example, when the particles are placed at the vertices of a regular $n$-gon). While Theorem~\ref{thm:main} ascertains those must be locally unstable, we still need to rule out existence of those serendipitous initial conditions that would lead to those limiting configurations. This is the content of another classical ingredient: the center-stable manifold theorem, which indeed shows that the limiting stationary point is almost always a stable one~\cite[Theorem III.7 and Exercise III.3]{shub}.

Putting these two ideas together, we get the following result that together with Theorem~\ref{thm:main} ascertain convergence to the synchronized states in $f_\beta$ dynamics with or without normalization.

\begin{lemma}[{\cite[Lemma A.1]{mathpersp25}}]
\label{lemma:converge}
Let $\cM$ be a compact Riemannian manifold and let $\mathsf{E}:\cM\rightarrow\mathbb{R}$ be a smooth function. The set of initial conditions $X_0\in\cM$ for which the gradient ascent system
\[
\begin{cases}
& \dot{\mathbf{x}}(t) = \nabla_\cM \mathsf{E}(\mathbf{x}(t)), \\
& \mathbf{x}(0) = X_0
\end{cases}
\]
converges to a critical point of $\mathsf{E}$ at which the Hessian of $\mathsf{E}$ has a positive eigenvalue is of volume zero.
\end{lemma}

With these preparations, we are ready to derive our main global synchronization results by
applying Lemma~\ref{lemma:converge} and Theorem~\ref{thm:main} to
systems~\eqref{eq:generalsystem2} with $f(x) = h(\cos(x))\sin(x)$. Indeed, for such systems we can
see that dynamics becomes a gradient ascent on the potential
\begin{equation}\label{eq:energy}
	E(\mathbf{x}) = \sum_{i,j} \phi(\cos(x_i-x_j))\,,\qquad \phi(t) = \int_0^t h(s) ds\,.
\end{equation}
Note that a critical point $\mathbf{x}$ is locally unstable if and only if the Hessian at
$\mathbf{x}$ has positive eigenvalue since the Hessian is symmetric, thus enabling application of
Theorem~\ref{thm:main}. See Theorem~\ref{thm:ascentnormal} shortly, for the full statement.

\subsection{Adjusted gradient ascent}

It turns out that the method discussed above is applicable not only to systems of the
type~\eqref{eq:generalsystem2}, but also to more general systems with particle-dependent
normalization factors, i.e. systems of the type~\eqref{eq:systemnormalized}. We need to consider
this extension because the the simplified model of
self-attention
(see~\eqref{eq:selfattentionnormal} below) has precisely such form.

\begin{theorem}
\label{thm:ascentnormal}
Assume that $f(x)=\sin(x)h(\cos(x))$, where $h$ is a real-analytic function on an open set containing $[-1,1]$ and $\tau\int_{-\pi}^\pi |f'''(x)|_+ dx\leq 4\left(1+\frac{\tau}{2\pi}\right)f'(0)$, where $\tau$ is as in Theorem~\ref{thm:main}. Then, global synchronization occurs in \pref{eq:systemnormalized}.
\end{theorem}

The full proof of this result is given in Section~\ref{sec:normalized} below, but the idea is
simple. First, when $f(x)=\sin(x)h(\cos(x))$ and normalization factors $g_i=1$, then as we have
seen in the previous section we are dealing with a gradient ascent on the
potential~\eqref{eq:energy}, which thus must converge (generically) to a locally stable critical
point. In the case of $g_i\not= 1$, we can follow the idea suggested in \cite[Section 3.4 and
Remark B.1]{mathpersp25}: by introducing a non-flat Riemannian metric on $(\mathbb{S}^1)^{\otimes n}$ we
can make sure that the gradient of the same energy~\eqref{eq:energy} results in the normalized
dynamics~\eqref{eq:systemnormalized}. Then, this case also reduces to an application of
Theorem~\ref{thm:main}.

We mention that we further generalize the last result to systems where each particle contributes
to the RHS in~\eqref{eq:systemnormalized} with its own, particle-dependent weight factor. See
Section~\ref{sec:generalizedsystem} for more.

\subsection{Self-attention dynamics}
\label{subsec:selfattention}

Next, we discuss an application of the main results to self-attention dynamics. Recall that the latter~\cite[(USA)]{mathpersp25} is defined as
\begin{equation}\tag{T1}
\label{eq:selfattention}
\dot{x}_i(t) = -\sum_{j=1}^n e^{\beta\cos(x_i(t)-x_j(t))}\sin(x_i(t)-x_j(t)), \, 1\leq i\leq n,
\end{equation}
which corresponds to taking $f(x)=f_\beta(x)=\sin(x) e^{\beta \cos(x)}$. 
The global synchronization conjectured in~\cite{mathpersp25} for all $\beta \ge 0$ was only shown for $\beta \le 1$ and $\beta \ge \Omega(1/n)$, cf.~\cite{synchronization_nonlinear}. In this section we resolve the conjecture in full and in fact even extend it to a portion of $\beta < 0$. 

Define the number 
$$
a(\beta) = \inf\{\tau: f_\beta'(x) < 0\,, \  \forall\, x\in(\tau,\pi], \tau \in [0,\pi]\}\,. 
$$
Theorem~\ref{thm:ascentnormal} implies that whenever
$$ 4 {1+ {a(\beta) \over 2\pi} \over a(\beta) \int_{-\pi}^\pi |f'''(x)|_+ dx} > 1,$$
global synchronization occurs. In fact, using specific properties of $f_\beta$ we can strengthen the criterion in Theorem slightly, see Corollary~\ref{cor:main} below, and guarantee global synchronization under the weaker assumption of 
$$ 4 {1+ {a(\beta) \over \pi} \over a(\beta) \int_{-\pi}^\pi |f'''(x)|_+ dx} > 1\,.$$
The quantity on the left-hand side is termed the \textit{synchronization ratio} and we numerically plot it on Fig.~\ref{fig:synchronization}. As one can see the criterion indeed is verified in the region of $\beta \geq -0.25$. We formally verify the inequality in a slightly narrower region of $\beta \geq -0.16$ below.

\begin{figure}[!h]
\centering
\includegraphics[width=\linewidth]{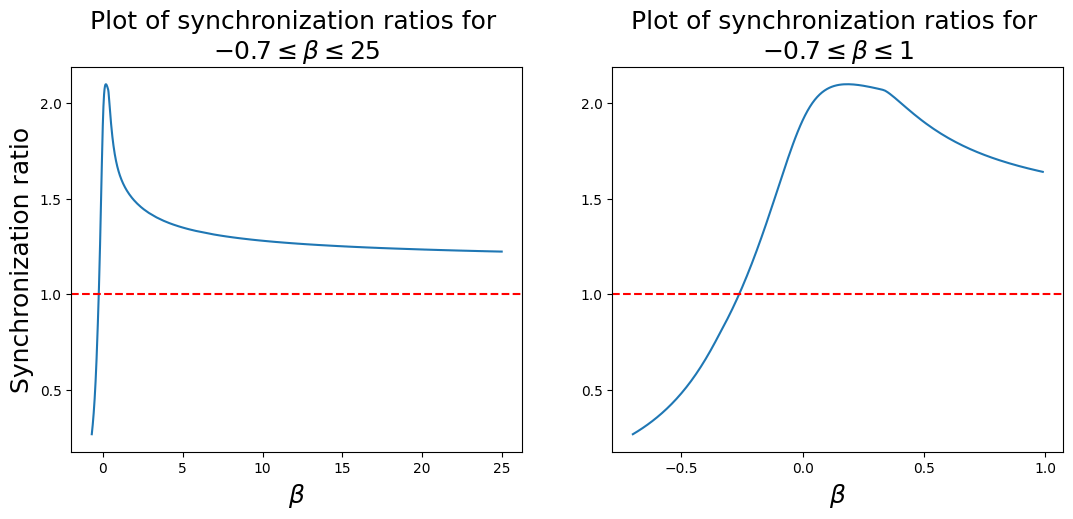}
\caption{The figure plots the synchronization ratio $ \product{|f'''|_+, \tau}_{L_2}^{-1}4\left(1+\frac{\tau}{\pi}\right) f'(0)$ from \Cref{cor:main} with $f(x)$ set as $\sin(x)e^{\beta(\cos(x)-1)}$ and $M$ set as $\pi$. A ratio greater than one indicates that we have determined that global synchronization occurs.}
\label{fig:synchronization}
\end{figure}

\begin{corollary}
\label{cor:conjecture}
Suppose $\beta\geq -0.16$. Then, global synchronization occurs in \pref{eq:selfattention}.
\end{corollary}

The dynamics~\eqref{eq:selfattention} is a simplification of the actual self-attention dynamics, which is given by~\cite[(SA)]{mathpersp25}:
\begin{equation}\tag{T2}
\label{eq:selfattentionnormal}
\dot{x}_i(t) = -\frac{1}{\sum_{j=1}^ne^{\beta\cos(x_i(t)-x_j(t))}}\sum_{j=1}^n e^{\beta\cos(x_i(t)-x_j(t))}\sin(x_i(t)-x_j(t)), \, 1\leq i\leq n.
\end{equation}

As already explained in the previous section (Theorem~\ref{thm:ascentnormal}), the results about unnormalized system can be easily transported to results about the normalized system, which allows us to conclude with the following:
\begin{corollary}
\label{cor:conjecturenormal}
Suppose $\beta\geq -0.16$. Then, global synchronization occurs in \pref{eq:selfattentionnormal}.
\end{corollary}

Finally, one might wonder whether global synchronization occurs for even more negative values of $\beta$. We show that this is not the case, thus leaving only the region $\beta \in \left(-\frac{2}{3},-0.25\right)$ in uncertain synchronization status.

\begin{corollary}
\label{cor:negativebeta}
Suppose $\beta<-\frac{2}{3}$. There exists a constant $C_\beta$ such that if $n$ is divisible by $3$ or $n\geq C_\beta$, then global synchronization does not occur in either~\eqref{eq:selfattention} or~\eqref{eq:selfattentionnormal}.
\end{corollary}

The proofs of all results can be found in \Cref{sec:application}.

We remark that for $\beta > 0$ self-attention dynamics (normalized or not) corresponds to gradient
ascent on the potential
\begin{equation}\label{eq:energy_sad}
	E(\mathbf{x}) = {1\over \beta} \sum_{i,j} e^{\beta \cos (x_i-x_j)}\,.
\end{equation}
For $\beta < 0$, self-attention dynamics is a \textit{gradient descent} on the potential
$$ E(\mathbf{x}) = {1\over |\beta|} \sum_{i,j} e^{-|\beta| \cos (x_i-x_j)}\,.$$
In either case, the global optimizer is clearly the synchronized configuration. However, in the
latter case local extrema with non-zero volume basin of attraction may emerge for large $|\beta|$.

Finally, we remark that \textit{gradient descent} on the potential~\eqref{eq:energy_sad} with
$\beta>0$ yields a completely different dynamics, corresponding to taking
$f(x)=-\sin(x)e^{\beta\cos(x)}$ in~\eqref{eq:generalsystem2}. In this case
the particles tend to equi-disperse on the circle. Indeed, the unique global minimizer
of~\eqref{eq:energy_sad} is the $n$-gon as shown in \cite{cohn2007universally}. At the level of the evolution of measures, the
unique global minimizer of the functions $\mu \mapsto \int \int e^{\beta \cos(x-y)} \mu(dx)\mu(dy)$ can be
easily seen to be the uniform measure (e.g. by noticing that Fourier coefficients of $e^{\beta
\cos(x)}$ are all positive, cf.~\cite[Section 7.2]{mathpersp25}), thus explaining the equidistribution
tendency.

\section{Proof of \texorpdfstring{\Cref{thm:main}}{}}
\label{sec:mainproof}

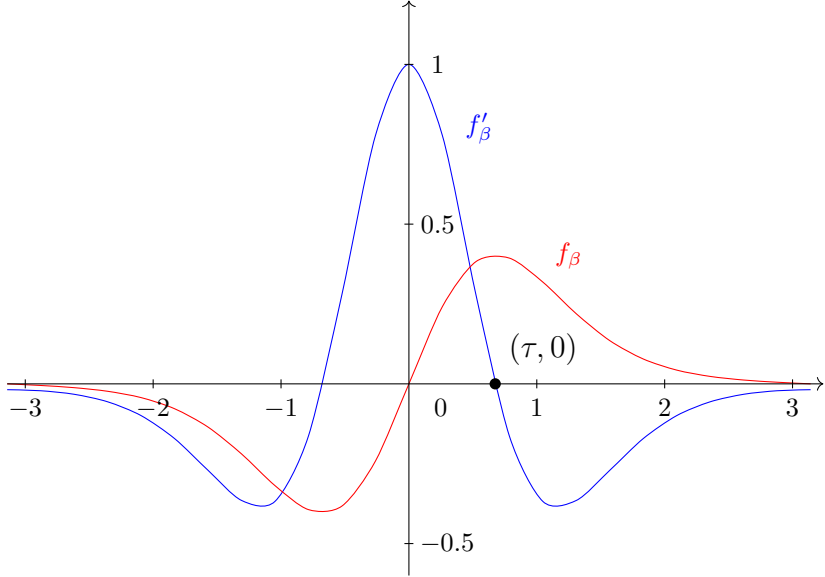
\begin{figure}
\begin{center}
\begin{tikzpicture}[scale=0.85]

\draw[xscale=2, yscale = 5, domain=-pi:pi, smooth, variable=\x, blue] plot ({\x}, {(cos(\x*180/pi) - 2*sin(\x*180/pi)*sin(\x*180/pi))*exp(2*cos(\x*180/pi)-2)});

\draw[xscale=2, yscale = 5, domain=-pi:pi, smooth, variable=\x, red] plot ({\x}, {sin(\x*180/pi)*exp(2*cos(\x*180/pi)-2)});

\draw[->] (-2*pi,0) -- (2*pi+0.2,0);

\draw[->] (0,-3) -- (0,6);

\foreach \x in {-3,-2,-1,1,2,3}
\draw[shift={(2*\x,0)},color=black] (0pt,2pt) -- (0pt,-2pt) node[below] {\footnotesize $\x$};

\node[below] at (0.5,-2pt) {\footnotesize $0$};

\node[color = blue] at (1.1, 4) {\footnotesize $f'_\beta$};

\node[color = red] at (2.5, 2) {\footnotesize $f_\beta$};

\fill (4*0.33744, 0) circle[radius=2.5pt];
\node at (4*0.33744+0.75, 0.55) {$(\tau,0)$};

\foreach \y in {-0.5,0.5, 1}
\draw[shift={(0,5*\y)},color=black] (2pt,0pt) -- (-2pt,0pt);

\foreach \y in {-0.5} \node at (0.6, 5*\y) {\footnotesize $\y$};

\foreach \y in {0.5,1} \node at (0.45, 5*\y) {\footnotesize $\y$};

\end{tikzpicture}
\end{center}
\caption{The function $f_\beta$ (red) and its derivative $f'_\beta$ (blue) for $\beta=2$. The parameter $\tau=\tau(\beta)$ is defined as the unique solution to the equation $f'_\beta(\tau)=0$ over $[0, \pi]$. For $\beta=2$, $\tau\simeq 0.6749$.
}
\label{fig:ampot}
\end{figure}

Suppose ${\bf x} = (x_i)_{1\leq i\leq n} \in \torus^n$. For the system~\eqref{eq:generalsystem2}, a point is \textit{stationary} iff
\begin{equation}\tag{C1}
\label{eq:condition1}
\sum_{j=1}^n f(x_i-x_j)=0, \, \forall i\in[n].
\end{equation}
In view of Lemma~\ref{lemma:converge}, we also need a condition for the point $\bf x$ to be stable.
In fact, we  only need a weaker condition that is implied by stability, which has been the basis of
proving synchronization in Kuramoto-type dynamics since its introduction
in~\cite[(2.4)]{Taylor2012}. We say that point $\bf x$ is \textit{cut-stable} if
\begin{equation}\tag{C2}
\label{eq:condition2}
\sum_{i\in S,\,j\in S^C} f'(x_i - x_j) \geq 0
\end{equation} 
for all $S\subset [n]$ such that the value of $x_i$ is the same for all $i\in S$.

Our proof will show that any stationary point satisfying \eqref{eq:condition2} must be synchronized. In other words, we will show that for any non-synchronized stationary point $\mathbf{x}$ there exists a set $S$ defining an ``escape direction'' for the linearization of the system~\eqref{eq:generalsystem2}, corresponding to moving points in $S$ clockwise and points in $S^c$ counter-clockwise at the same speed.\footnote{Note that when state space is $\mathbb{S}^d$ with $d>1$, then finding escape directions in mean-field systems can be done by pulling all particles toward the same direction (subject to spherical constraints), cf.~\cite{markdahl} and~\cite{synchronization_nonlinear}. Our method is  more involved.}

We start with a review of the proof from~\cite{mathpersp25} which applies to $f(x)=f_\beta(x) = e^{\beta\cos(x)}\sin(x)$, whose derivative $f_\beta'$ when $\beta=2$ is shown in Fig.~\ref{fig:ampot}, where the crucial parameter $\tau$ from Theorem~\ref{thm:main} is also shown. If we apply~\eqref{eq:condition2} with $S=\{1\}$ then we see that there must be at least one particle, say $2$, at distance $\le \tau$ from 1 (otherwise all $f'(x_1-x_j)<0$). We can now apply the argument to $S=\{1,2\}$ to find that $3$ must be at distance $\le \tau$, etc. Overall, if $n\tau < \pi$ then all particles must be inside one half-circle. But then if $i_0$ is the boundary particle, then~\eqref{eq:condition1} with $i=i_0$ implies that all $x_j=x_{i_0}$, because $f(x_{i_0}-x_j)\ge 0$. Unfortunately, this proof only shows synchronization when $\pi > n\tau \asymp {n\over \sqrt{\beta}}$. Our contribution here is a method that extends to arbitrary large $n$.

To describe our idea, let us define the vector field acting on particles as
$$ \chi(x) = \sum_{j} f(x-x_j)\,.$$
Then from~\eqref{eq:condition1} and~\eqref{eq:condition2} (applied with $S=\{i\}$) we know that 
$$ \chi(x_i) = 0, \qquad \chi'(x_i) \ge f'(0)\,.$$
Consider a pair of adjacent particles $x_{i} < x_{i+1}$. Because $\chi(x_i)=\chi(x_{i+1})$, we have from integration by parts that
$$ \chi'(x_i) + \chi'(x_{i+1}) = \int_{x_i}^{x_{i+1}} \chi{'''}(x) {(x-x_i)(x_{i+1}-x)\over x_{i+1}-x_i}dx\,.$$
As we have shown above, all intervals $x_{i+1} - x_i$ except possibly 1 are bounded by $\tau$. In the special case when \textit{all of them} are bounded by $\tau$ we can notice that the factor multiplying $\chi'''(x)$ is positive and upper-bounded by $\tau/2$. Thus summing over $i=1,\ldots,n-1$ we obtain
\begin{align*}
2 n f'(0) \le 2 \sum_i \chi'(x_i) & =  \sum_{i=1}^{n-1} \int_{x_i}^{x_{i+1}} \chi{'''}(x) {(x-x_i)(x_{i+1}-x)\over (x_{i+1}-x_i)} dx \\ & \le {\tau\over 2} \int |\chi'''(x)|_+ dx \le {n\tau\over 2} \int |f'''|_+ dx\,.
\end{align*}
This inequality, however, is not possible if (as is the case for $f=f_\beta$ for large $\beta$) we have $\tau \int|f'''|_+ < 4f'(0)$. Consequently, one of the assumptions must be violated.  The full proof below will show that in fact this contradiction implies that $x_1=\ldots=x_n$.

We proceed to the formal proof of Theorem~\ref{thm:main} and consider ${\bf x}=(x_1,\ldots,x_n)$ satisfying conditions~\eqref{eq:condition1}  and~\eqref{eq:condition2}. Denote the distinct values of the $x_i$ as $0\leq \theta_1 < \theta_2 < \cdots < \theta_K < 2\pi$ and let $\theta_{K+1} = 2\pi + \theta_1$. For $1\leq j\leq K$, we define $\Delta_j = \theta_{j+1} -\theta_j$, where we set $\Delta_K = 2\pi + \theta_1 - \theta_K$ to account for the periodic boundary. We refer to the $\Delta_j$ as the \textit{gaps}. Furthermore, we let $\tau_{\text{max}}(\bf x)$ denote the maximum of $\Delta_j$ for $1\leq j\leq K$.

\begin{lemma}
\label{lemma:twoconsecutive_general}
Assume that \pref{eq:condition2} is satisfied at $(x_1, \ldots, x_n)$. There do not exist distinct gaps $\omega_1$ and $\omega_2$ such that $\omega_1, \omega_2 >\tau$.
\end{lemma}

\begin{proof}
The idea is the same as \cite[Appendix B]{mathpersp25}. Suppose that the line $\ell$ intersects the interiors of both $\omega_1$ and $\omega_2$. Let $S$ be the set of $i$ such that $x_i$ is on one side of $\ell$. Then, for all $i\in S$ and $j\in S^C$, we have that $x_i-x_j \notin (-\min(\omega_1, \omega_2), \min(\omega_1, \omega_2))  \pmod{2\pi}$. Because $\min(\omega_1,\omega_2)>\tau$, $x_i-x_j\notin [-\tau, \tau] \pmod{2\pi}$, which implies that $f'(x_i-x_j) < 0$. This is a contradiction to \pref{eq:condition2}.
\end{proof}

For $i\in [K]$, let $N_i$ be the multiplicity of $\theta_i$, i.e. the number of $j\in [n]$ such that $x_j=\theta_i$. Also, let 
$$
\product{f, g}_{L_2} \triangleq \int_{-\pi}^\pi f(x)g(x) dx\,.
$$

\begin{lemma}
\label{lemma:taumax}
Suppose $(x_1, \ldots, x_n)\in\mathbb{T}^n$ is stationary \pref{eq:condition1} and cut-stable \pref{eq:condition2} for the system \eqref{eq:generalsystem2}. Furthermore, assume that $(x_1,\ldots,x_n)$ is not synchronized. Then,
\[
\product{1, |f'''|_+}_{L_2} \geq \min\left(\frac{8}{\tau}, \frac{4}{\tau}+\frac{4}{\tau_{\text{max}}(x_1,\ldots,x_n)}\right)f'(0).
\]
\end{lemma}

\begin{proof}
For the sake of contradiction, assume that $K>1$. Let
\[
\varphi(x) = \sum_{j=1}^n f'(x-x_j).
\]
Then, after using \pref{eq:condition2} with $S$ equal to the set of $j\in [n]$ such that $x_j=\theta_i$,
\begin{equation}
\label{eq:varphival_general}
\varphi(\theta_i) = N_if'(0) + \sum_{j\in [n],\,x_j\not= \theta_i} f'(\theta_i-x_j) \geq N_if'(0).
\end{equation}

Let 
\[
\Psi = \sum_{i=1}^K \mathbf{1}\{x\in (\theta_i, \theta_{i+1})\} \Psi_i,
\]
where $\Psi_i: [\theta_i, \theta_{i+1}] \rightarrow \mathbb{R}$ is twice differentiable. Then, for $i\in [K]$,
\[
\product{\Psi_i, \varphi''}_{L_2} = \Psi_i\varphi'\big|_{\theta_i}^{\theta_{i+1}}- \product{\Psi_i', \varphi'}_{L_2} = (\Psi_i\varphi' - \Psi_i'\varphi)\big|_{\theta_i}^{\theta_{i+1}} + \product{\Psi_i'', \varphi}_{L_2}.
\]

Assume that $\Psi_i''=-a_i$ for a constant $a_i$ and $\Psi_i(\theta_i)=\Psi_i(\theta_{i+1}) = 0$ over $[\theta_i, \theta_{i+1}]$, or equivalently that $\Psi_i(x)=\frac{a_i}{2}(x-\theta_i)(\theta_{i+1}-x)$. Then, we have that
\begin{align*}
\product{\Psi_i, \varphi''}_{L_2} &= -\Psi_i'(\theta_{i+1})\varphi(\theta_{i+1}) + \Psi_i'(\theta_i)\varphi(\theta_i) - a_i\product{\mathbf{1}\{(\theta_i, \theta_{i+1})\}, \varphi}_{L_2} \\ &= \frac{a_i}{2}\Delta_i(\varphi(\theta_i) + \varphi(\theta_{i+1})),
\end{align*}
since 
\[
\int_{\theta_i}^{\theta_{i+1}} \varphi(x) dx = \sum_{j=1}^n f(\theta_{i+1} - x_j) - \sum_{j=1}^n f(\theta_i - x_j) = 0.
\]
Hence, 
\[
\product{\Psi, \varphi''}_{L_2} = \sum_{i=1}^K \frac{a_i}{2}\Delta_i(\varphi(\theta_i) + \varphi(\theta_{i+1})).
\]

In particular, assuming that $a_i\geq 0$ for all $i\in [K]$, using \pref{eq:varphival_general} gives that
\[
\product{\Psi, \varphi''}_{L_2} \geq \sum_{i=1}^K \frac{a_i}{2}\Delta_i(N_i+N_{i+1})f'(0).
\]
Furthermore, since $\Psi_i\in[0,\frac{a_i\Delta_i^2}{8}]$, we have that
\[
\product{\sum_{i=1}^K a_i\Delta_i^2\mathbf{1}\{(\theta_i, \theta_{i+1})\}, |\varphi''|_+}_{L_2}\geq 4\sum_{i=1}^K a_i\Delta_i(N_i+N_{i+1})f'(0).
\]

Since  $\sum_{j=1}^n |f'''(x-x_j)|_+ \geq |\varphi''|_+$, we have that
\[
\sum_{j=1}^n \product{\sum_{i=1}^K a_i\Delta_i^2\mathbf{1}\{(\theta_i, \theta_{i+1})\}, |f'''(x- x_j)|_+}_{L_2}\geq 4\sum_{i=1}^K a_i\Delta_i(N_i+N_{i+1})f'(0),
\]
or equivalently,
\begin{equation}
\label{eq:prodineq_general}
\sum_{j=1}^n \product{\sum_{i=1}^K a_i\Delta_i\mathbf{1}\{(\theta_i, \theta_{i+1})\}, |f'''(x-x_j)|_+}_{L_2}\geq 4\sum_{i=1}^K a_i(N_i+N_{i+1})f'(0).
\end{equation}

By \Cref{lemma:twoconsecutive_general}, we can have that $\tau_{\text{max}}(x_1,\ldots,x_n)$ is greater than $\tau$, but all other gaps must be at most $\tau$. For brevity, we use $\tau_{\text{max}}$ to denote $\tau_{\text{max}}(x_1,\ldots,x_n)$ in the remainder of the proof.

In \pref{eq:prodineq_general}, set $a_i = \frac{1}{\Delta_i}$ to obtain
\begin{align*}
& \sum_{j=1}^n \product{\sum_{i=1}^K \mathbf{1}\{(\theta_i, \theta_{i+1})\}, |f'''(x-x_j)|_+}_{L_2} \geq 4\sum_{i=1}^K \frac{N_i+N_{i+1}}{\Delta_i}f'(0) \\
& \Leftrightarrow n\product{1, |f'''|_+}_{L_2} \geq 4\sum_{i=1}^K \frac{N_i+N_{i+1}}{\Delta_i} f'(0).
\end{align*}
Assume that the gap $\tau_{\text{max}}$ is between $\theta_\ell$ and $\theta_{\ell + 1}$. Since $\Delta_i \leq \tau$ for $i\in [K]\backslash \{\ell\}$, 
\begin{align*}
n\product{1, |f'''|_+}_{L_2} & \geq 4\left(\sum_{i\in[K]\backslash\{\ell\}} \frac{N_i+N_{i+1}}{\tau} + \frac{N_\ell + N_{\ell+1}}{\tau_{\text{max}}}\right)f'(0) \\
& = 4\left(\frac{2n}{\tau} - (N_\ell + N_{\ell+1})\cdot \left(\frac{1}{\tau} - \frac{1}{\tau_{\text{max}}}\right)\right) f'(0).
\end{align*}
The result follows after noting that $N_\ell+N_{\ell+1}\leq n$ because $K\geq 2$.
\end{proof}

\begin{corollary}
\label{cor:main}
Let $\tau\in(0,\pi]$ be such that $f'(x) < 0$ for all $x\not\in[-\tau,\tau]$. Assume that $M$ is a positive real number such that for all stable, stationary, and non-synchronized points $\bf x$ of \pref{eq:generalsystem2}, $\tau_{\text{max}}({\bf x}) < M$. If
\[
\tau\int_{-\pi}^\pi |f'''(x)|_+ dx\leq 4\left(1+\frac{\tau}{M}\right)f'(0),
\]
then every stable stationary point $(x_1, \ldots, x_n)$ of system~\eqref{eq:generalsystem2} on $\mathbb{T}^n$ is synchronized, i.e. $x_1=\cdots=x_n$.
\end{corollary}

Note that \Cref{cor:main} is more general than Theorem~\ref{thm:main} since $M=2\pi$ is always a valid choice.

\begin{proof}
From \Cref{lemma:taumax}, if ${\bf x} = (x_1,\ldots,x_n)$ is not synchronized, then
\[
\product{1, |f'''|_+}_{L_2} \geq \min\left(\frac{8}{\tau}, \frac{4}{\tau}+\frac{4}{\tau_{\text{max}}({\bf x})}\right)f'(0) > 4\left(\frac{1}{\tau} + \frac{1}{M}\right) f'(0)
\]
because $\tau_{\text{max}}({\bf x})<M$, which is a contradiction.
\end{proof}

\begin{remark}
In many cases, such as the Kuramoto model and self-attention dynamics, it is straightforward to show that $\tau_{\text{max}}({\bf x})<\pi$ for ${\bf x}$ to be stable, stationary, and non-synchronized, which leads to an improvement upon \Cref{thm:main} by setting $M=\pi$ in \Cref{cor:main}. For examples of such results, see \cite[Lemma 6.4]{mathpersp25} and \cite[Lemma 10]{synchronization_nonlinear}.
\end{remark}

\begin{corollary}
\label{corr:variableswitch}
Assume that $f'(0) \geq 0$ and $\tau\product{1, |f'''|_+}_{L_2} \leq 4\left(1+\frac{\tau}{2\pi}\right)f'(0)$. Suppose $a,b\in\mathbb{R}$ satisfy $ab\geq 0$ and $a$ and $b$ are not both zero. Assume that 
\begin{equation}
\label{eq:condition3}
\sum_{j=1}^n af(x_i-x_j)-bf(x_j-x_i) = 0, \, \forall i \in [n]
\end{equation}
and \pref{eq:condition2} are satisfied at $(x_1, \ldots, x_n)$. Then, $x_i=x_j$ for all $i,j\in[n]$.
\end{corollary}

\begin{proof}
Let $g(x)=|a|f(x)-|b|f(-x)$. Then, $g'(x)=|a|f'(x) + |b|f'(-x)$ and $g'''(x) = |a|f'''(x)+|b|f'''(-x)$. We have that $g$ satisfies \pref{eq:condition1} and \pref{eq:condition2} at $(x_1,\ldots,x_n)$. Furthermore, $g'(0)=(|a|+|b|)f'(0) \geq 0$ and $|g'''(x)|_+ \leq |a||f'''(x)|_+ + |b||f'''(-x)|_+$ so $\tau \product{1, |g'''|_+}_{L_2} \leq \tau (|a|+|b|)\product{1, |f'''|_+}_{L_2}\leq 4\left(1+\frac{\tau}{2\pi}\right)g'(0)$. The result follows from \Cref{thm:main}.
\end{proof}

\section{Gradient ascent systems}
\label{sec:normalized}

The goal of this section is to prove \Cref{thm:ascentnormal}, which characterizes the asymptotic
behavior of (adjusted) mean-field gradient ascent systems. Following the statement of the theorem, assume that $f(x)=\sin(x)h(\cos(x))$ for some real-analytic function $h$ on an open set containing $[-1,1]$; recall that $g: (\mathbb{S}^1)^n\rightarrow \mathbb{R}_{>0}$ is smooth. 

As explained in \Cref{lemma:converge}, gradient ascent almost always converges to a stationary point. The key idea of the proof is to use \Cref{lemma:converge} to argue that the stationary point almost always has negative semi-definite Hessian and then apply \Cref{thm:main} to characterize the stationary point as being synchronized. As part of the proof, we show that a stationary point with negative semi-definite Hessian satisfies conditions \pref{eq:condition1} and \pref{eq:condition2} so that we can apply \Cref{thm:main}. 

We rewrite \pref{eq:systemnormalized} in terms of points on the manifold $(\mathbb{S}^1)^n$. Note that this is the setting that \cite{mathpersp25,metastability} originally consider. For $i\in [n]$ and $t\geq 0$, let 
\[
\p_i(t)=(\cos(x_i(t)),\,\sin(x_i(t)))\in \mathbb{S}^1.
\]
For a point $\p_i \in \mathbb{S}^1$ let us introduce $\p_i^\perp$ to be (the unique) positively oriented unit
vector in the tangent space at $\p_i$. Let us denote by $P_{\p_i^\perp}(\cdot)$ the linear operator
orthogonally projecting $\mathbb{R}^2$ onto the span of $\p_i^\perp$. With this notation we can
rewrite~\eqref{eq:systemnormalized} as
\begin{equation}\tag{S2'}
\label{eq:systemnormalized2}
\dot{\p}_i(t) = \frac{1}{g_i(\p(t))} \sum_{j=1}^n h(\product{\p_i(t), \p_j(t)}) P_{\p_i(t)^\perp}(\p_j(t))\, \forall i\in [n],
\end{equation}
where $\p(t)=(\p_1(t),\ldots,\p_n(t))\in (\mathbb{S}^1)^n$ for $t\geq 0$. To see the equivalence,
note that $P_{\p_i(t)^\perp}\p_j(t)=-\sin(x_i(t)-x_j(t))\p_i(t)^\perp$ and $\product{\p_i(t), \p_j(t)}=\cos(x_i(t)-x_j(t))$.

Next, we rewrite \pref{eq:systemnormalized2} as the gradient of a function over a Riemannian manifold. Suppose $\varphi(x) = \int_0^x h(x) dx$, which is also an analytic function. Let
\[
E(x_1,\ldots,x_n) = \frac{1}{2}\sum_{i,j=1}^n \varphi(\product{x_i, x_j}).
\]
In order to describe the system as gradient ascent, we construct a Riemannian manifold such that the gradient computed with respect to its metric is the dynamics of \pref{eq:systemnormalized2}. We follow the ideas of \cite[Section 3.4]{mathpersp25} to state and prove the following lemma.

\begin{lemma}
\label{lemma:riemannian}
Suppose the Riemannian manifold $M$ over $(\mathbb{S}^1)^n$ has positive-definite inner product $\product{(a_1, \ldots, a_n), (b_1, \ldots, b_n)}_P = \sum_{i=1}^n \alpha_i(P)a_ib_i$ at $P=(p_1,\ldots,p_n) \in(\mathbb{S}^1)^n$. Then,
\[
(\nabla_M)_i E(P) = \frac{1}{\alpha_i(P)} \sum_{j=1}^n h(\product{p_i, p_j}) P_{p_i^\perp}(p_j)\,\forall i\in[n].
\]
\end{lemma}

\begin{proof}
Let $Y$ be a vector field over $(\mathbb{S}^1)^n$ with gradient flow $\Phi_Y^t$. Furthermore, suppose the vector field $B$ satisfies
\[
B_i(P) = \frac{1}{\alpha_i(P)} \sum_{j=1}^n h(\product{p_i, p_j}) P_{P_i^\perp}(P_j)
\]
for $i\in [n]$. Then, it suffices to prove that 
\[
\frac{d}{dt}\bigg|_{t=0} E(\Phi_Y^t(P)) = \product{Y(P), B(P)}_P.
\]

By considering a linear basis over $T_P(\mathbb{S}^1)^n$, we only need to show this holds when $Y(P)=(Ap_1, 0, \ldots, 0)$ for a non-zero skew symmetric $A$. In this case, $\Phi_Y^t(P) = (e^{At}p_1, p_2, \\ \ldots, p_n)$, so
\[
E(\Phi_Y^t(P)) = \sum_{j\not=1} \varphi(\product{e^{At}p_1, p_j}) + \frac{1}{2}\varphi(\product{e^{At}p_1, e^{At}p_1}) + C,
\]
where $C$ does not depend on $t$. Thus,
\begin{align*}
& \frac{d}{dt} E(\Phi_Y^t(P)) = \sum_{j\not=1} h(\product{e^{At}p_1, p_j})\product{Ae^{At}p_1, p_j} + h(\product{e^{At}p_1, e^{At}p_1}) \product{Ae^{At}p_1, p_1} \\
& \Rightarrow \frac{d}{dt}\bigg|_{t=0} E(\Phi_Y^t(P)) = \sum_{j=1}^n h(\product{p_1, p_j}) \product{Ap_1, p_j}.
\end{align*}
It suffices to prove that 
\[
\sum_{j=1}^n h(\product{p_1, p_j}) \product{Ap_1, p_j} = \product{Y(P), B(P)}_P = \product{Ap_1, \sum_{j=1}^n h(\product{p_1, p_j}) P_{p_1^\perp}(p_j)}.
\]
Hence, it suffices to prove that for all $v\in\mathbb{S}^1$,
\[
\product{Ap_1, v} = \product{Ap_1, P_{p_1^\perp}(v)} = \product{Ap_1, v - \product{v, p_1}p_1},
\]
which would be implied by $\product{Ap_1, p_1} = 0$. Observe that $\product{Ap_1, p_1} = p_1^{\top}Ap_1$, which equals $0$ because $A$ is skew-symmetric.
\end{proof}

\Cref{lemma:converge} implies the almost always convergence to a critical point of $E$ with a negative semidefinite Hessian. However, as explained in \cite[Remark B.1]{mathpersp25}, when analyzing whether the Hessian is negative semidefinite at a critical point, any two metrics are equivalent. We state this classical idea in the following lemma, and afterwards, we use it to prove the main result.

\begin{lemma}
\label{lemma:hessian}
Suppose $R_1=((\mathbb{S}^{d-1})^n, g_1)$ and $R_2=((\mathbb{S}^{d-1})^n, g_2)$ are Riemannian
manifolds. Let $P$ be a critical point of the analytic function $\gamma:(\mathbb{S}^{d-1})^n
\rightarrow \mathbb{R}$. For $1\leq i\leq 2$, let $H_i$ be the Hessian of $\gamma$ at $P$ with
respect to $R_i$.
Suppose $v\in T_P((\mathbb{S}^{d-1})^n)$. Then, $\langle H_1v, v\rangle_{R_1} > 0 \Leftrightarrow
\langle H_2 v, v\rangle_{R_2} > 0$.
\end{lemma}

\begin{remark}
The set of critical points for both metrics are the same. We abuse notation and assume that
$H_i$ is a matrix expressing the Hessian in terms of an orthonormal basis for the metric $R_i$ at
$P$, and thus write $v^{\top}H_iv$ instead of $\langle H_i v,v\rangle_{R_i}$, implying that
$v \in T_P$ is itself expressed as a column vector in the orthonormal basis with respect to $R_i$
at $P$.
\end{remark}

\begin{proof}[Proof of \texorpdfstring{\Cref{thm:ascentnormal}}{}]

For this proof, we are working in the setting of points on $(\mathbb{S}^1)^n$ that we have introduced in this section. Let $R$ be the Riemannian manifold over $(\mathbb{S}^1)^n$ with positive-definite inner product $\product{(a_1, \ldots, a_n), (b_1, \ldots, b_n)}_P = \sum_{i=1}^n g_i(P)a_ib_i$ at $P$. Then, from \Cref{lemma:riemannian} with $\alpha_i=g_i$ for all $i\in [n]$,
\[
\dot{\p}(t) = \nabla_R E(\p(t))
\]
in \pref{eq:systemnormalized2}. Applying \Cref{lemma:converge} gives that we have almost-sure convergence to a critical point $P=(p_1,\ldots, p_n)$ with negative semi-definite Hessian. Then, if $H$ is the Hessian of $E$ at $P$ with respect to $R$ and expressed in terms of an orthonormal basis for the metric $\product{\cdot, \cdot}_P$, we assume that there does not exist $v\in T_P((\mathbb{S}^1)^n)$ such that $v^{\top}Hv>0$.

First, observe that for all $i\in [n]$,
\[
\sum_{j=1}^n h(\product{p_i, p_j})P_{p_i^\perp}(p_j) = 0
\]
by the definition of a critical point of \pref{eq:systemnormalized2}, which can in turn be
rewritten as $\sum_j h(\cos(x_i-x_j))\sin(x_j-x_i) = \sum_j f(x_j-x_i) = 0$, thus verifying condition \pref{eq:condition1}.

Let $S$ be the Riemannian manifold over $(\mathbb{S}^1)^n$ with the standard inner product. Let $H_S$ be the Hessian of $E$ at $P$ with respect to $S$ and expressed in terms of the standard orthonormal basis. By applying \Cref{lemma:hessian} with the Riemannian manifolds $R$ and $S$ and the function $E$ as $\gamma$, because there does not exist $v\in T_P(\mathbb{S}^1)^n)$ such that $v^{\top}Hv>0$, there does not exist $v\in T_P((\mathbb{S}^1)^n)$ such that $v^{\top}H_Sv>0$.

The next step is to show that $P$, when written as an element of $\mathbb{T}^n$, is cut-stable, cf.~\eqref{eq:condition2}, so that we can apply \Cref{thm:main}. For this purpose, we follow the method of \cite[Appendix A]{mathpersp25}. 

For $t\geq 0$, define
\[
p(t) = [e^{c_iBt}p_i]_{i\in [n]},
\]
where $B=\begin{pmatrix} 0 & -1\\
1 & 0\end{pmatrix}$ is skew-symmetric. First, observe that
\[
E(p(t)) = \sum_{i,j\in[n], i\not=j} \varphi(\product{e^{c_iBt}p_i, e^{c_jBt}p_j}) +C,
\]
where $C$ does not depend on $t$. 

Suppose we let $v=\frac{d}{dt}\big|_{t=0} p(t)$. Observe that $v_i=c_iBp_i$ for all $i\in [n]$, so $v\in T_P((\mathbb{S}^1)^n)$. Then, 
\begin{equation}
\label{eq:nonpositive}
\frac{d^2}{dt^2}\Big|_{t=0} E(p(t)) = v^{\top}H_Sv \leq 0.
\end{equation}

Next, where $h(x)=\frac{d}{dx}\varphi(x)$,
\begin{align*}
\frac{d}{dt}E(p(t)) = \sum_{i,j\in[n], i\not=j} h(\product{e^{c_iBt}p_i, e^{c_jBt}p_j})\left(\product{c_iBe^{c_iBt}p_i, e^{c_jBt}p_j} + \product{e^{c_iBt}p_i, c_jBe^{c_jBt}p_j}\right).
\end{align*}
Furthermore,
\begin{align*}
& \frac{d^2}{dt^2} E(p(t)) = \sum_{i,j\in[n],i\not=j} \\
& \bigg[ h'(\product{e^{c_iBt}p_i,e^{c_jBt}p_j)})\left(\product{c_iBe^{c_iBt}p_i, e^{c_jBt}p_j} + \product{e^{c_iBt}p_i, c_jBe^{c_jBt}p_j}\right)^2 \\& + h(\product{e^{c_iBt}p_i,e^{c_jBt}p_j)})\times  \bigg(\product{c_i^2B^2e^{c_iBt}p_i,e^{c_jBt}p_j} +\product{c_iBe^{c_iBt}p_i, c_jBe^{c_jBt}p_j} \\
& + \product{e^{c_iBt}p_i, c_j^2B^2e^{c_jBt}p_j}\bigg) \bigg]
\end{align*}

For $1\leq i\leq n$, let $x_i$ be the unique element of $\mathbb{T}$ such that $p_i=(\cos(x_i),\sin(x_i))$. Because $B^2=-I_2$, $\product{B^2p_i,p_j} = -\cos(x_i-x_j)$. Since $B$ is the rotation by $90^\circ$ matrix, $\product{Bp_i,p_j} = \cos(x_i+90-x_j) = \sin(x_j-x_i)$. Thus,
\begin{align*}
&\frac{d^2}{dt^2}\bigg|_{t=0} E(p(t)) \\ & =  \sum_{i,j\in[n],i\not=j} (c_i-c_j)^2(-\cos(x_i-x_j)h(\cos(x_i-x_j)) + \sin(x_i-x_j)^2h'(\cos(x_i-x_j))).
\end{align*}
Since $\frac{d^2}{dt^2}\big|_{t=0} E(p(t)) \leq 0$ by \pref{eq:nonpositive},
\[
\sum_{i,j\in[n],i\not=j} (c_i-c_j)^2(\cos(x_i-x_j)h(\cos(x_i-x_j)) - \sin(x_i-x_j)^2h'(\cos(x_i-x_j))) \geq 0.
\]
Observe that because $f(x)=\sin(x)h(\cos(x))$, $f'(x) = \cos(x)h(\cos(x)) - \sin(x)^2h'(\cos(x))$. Therefore,
\[
\sum_{i,j\in[n]} (c_i-c_j)^2f'(x_i-x_j) \geq 0,
\]
so \pref{eq:condition2} is satisfied by setting $c_i=\mathbf{1}\{i\in S\}$ for $S\subset [n]$. Then, by \Cref{thm:main}, the $x_i$ are synchronized, which finishes the proof.
\end{proof}

\begin{remark}
\label{remark:negative_semidef}
Observe that we have shown that if the critical point $P$ has negative semi-definite Hessian, then 
\[
\sum_{i,j\in[n]} (c_i-c_j)^2f'(x_i-x_j)\geq0
\]
for all $c_i,c_j\in\mathbb{R}$. This result is well-known, but we include the computations for completeness. The other direction is true as well, since for any $v\in T_P((\mathbb{S}^1)^n)$, we have that $v_i=c_iBp_i$ for some $c_i\in\mathbb{R}$ for all $i\in [n]$.
\end{remark}

\section{Application to self-attention dynamics}
\label{sec:application}

In this section, we prove \Cref{cor:conjecture}, \Cref{cor:conjecturenormal} and
\Cref{cor:negativebeta}.

\begin{proof}[Proof of \texorpdfstring{\Cref{cor:conjecture}}]

For $\beta\geq 0$, system \pref{eq:selfattention} is equivalent to \pref{eq:generalsystem2} with $f(x)=\sin(x)\\e^{\beta(\cos(x)-1)}$ and $g=1$. By \Cref{thm:ascentnormal}, it suffices to show that $\tau\product{1, |f'''|_+} < 4f'(0)$. We prove this by considering cases for $\beta$. 

First, observe that
\begin{align*}
& f'(x) = (\cos(x)-\beta\sin(x)^2)e^{\beta(\cos(x)-1)}, \\
& f''(x) = (-\sin(x)-3\beta\sin(x)\cos(x)+\beta^2\sin(x)^3)e^{\beta(\cos(x)-1)}, \\
& f'''(x) = -(\cos(x)+3\beta\cos(x)^2 - 4\beta\sin(x)^2 -6\beta^2\cos(x)\sin(x)^2 + \beta^3\sin(x)^4)e^{\beta(\cos(x)-1)}.
\end{align*}
We reference these expressions later. Furthermore, it is clear that we can set
\[
\tau = \arccos\left(\frac{\sqrt{1+4\beta^2}-1}{2\beta}\right).
\]

\smallskip
\textbf{Case of $\beta> \frac{1}{3}$.}
The positive regions of $f'''$ over $(-\pi, \pi]$ are $(-a, -b)\cup (b, a)$ for some $a,b$ such that $0<a<b<\pi$, see \Cref{lemma:positiveregion}. Since $f'''$ is even, it suffices to prove that 
\[
\product{\tau, \mathbf{1}\{(-a, -b)\cup (b, a)\}f'''} < 4\Leftrightarrow \tau(f''(a)-f''(b)) < 2,
\]
which follows from \Cref{lemma:beta1}, \Cref{lemma:beta_0.75}, \Cref{lemma:beta_0.5}, and \Cref{lemma:beta_0.33}.

\smallskip
\textbf{Case of $0 < \beta \leq \frac{1}{3}$.}
The positive region of $f'''$ over $[0, 2\pi)$ is $(a, 2\pi-a)$ for some $a\in (0, \frac{\pi}{2})$, see \Cref{lemma:positiveregion2}. Then, it suffices to prove that 
\[
\product{\tau, \mathbf{1}\{(a, 2\pi-a)\}f'''} < 4 \Leftrightarrow \tau f''(a) > -2,
\]
which is proved in \Cref{lemma:beta_final}.

\smallskip
\textbf{Case of $\beta=0$.}
This corresponds to the Kuramoto model. In this case, $f'(x) = \cos(x)$ and $\tau=\frac{\pi}{2}$, because if $x\in (\frac{\pi}{2}, \frac{3\pi}{2})$ then $f'(x)$ is negative. Then, $f''(x) = -\sin(x)$ and $f'''(x) = -\cos(x)$ has positive region $(\frac{\pi}{2}, \frac{3\pi}{2})$ in $[0, 2\pi)$. Using the notation for the $\beta\in(0, \frac{1}{3}]$ case, $a=\frac{\pi}{2}$ and it suffices to prove that $\tau f''(a)>-2$, which is true because $\tau f''(a) = -\frac{\pi}{2} > -2$.

\smallskip
\textbf{Case of $-0.16 \leq \beta < 0$.}
From \Cref{lemma:semicircle}, we may set $M=\pi$ in \Cref{cor:main}. Afterwards, we prove that global synchronization occurs in \pref{eq:selfattention} by applying \Cref{cor:main} with $M=\pi$ and $g_i=1$ for all $i$. By setting $g_i(x_1,\ldots,x_n)=\sum_{j=1}^n e^{\beta\cos(x_i-x_j)}$, we also show that global synchronization occurs in \pref{eq:selfattentionnormal}.

By \Cref{cor:main}, it suffices to prove that $\tau\product{1,|f'''|_+} \leq 4\left(1+\frac{\tau}{\pi}\right)$. The positive region of $f'''$ over $(-\pi, \pi]$ is $(a, 2\pi-a)$ for some $a\in (\frac{\pi}{2}, \pi)$, see \Cref{lemma:positiveregion3}. Since $f'''$ is even, it suffices to prove that
\[
\product{\tau, \mathbf{1}\{(a, 2\pi-a)\}f'''}\leq 4\left(1+\frac{\tau}{\pi}\right)\Leftrightarrow \tau f''(a) \geq 2\left(1+\frac{\tau}{\pi}\right),
\]
which follows from \Cref{lemma:beta_negative}.
\end{proof}

\medskip

\begin{proof}[Proof of \Cref{cor:conjecturenormal}]

This follows from the same argument as the proof of \Cref{cor:conjecture} but with $g_i(x_1,\ldots,x_n) = \sum_{j=1}^n e^{\beta\cos(x_i-x_j)}$ for $i\in [n]$ in \pref{eq:systemnormalized}, which is a smooth function.
\end{proof}

\medskip

\begin{proof}[Proof of \Cref{cor:negativebeta}]

From \Cref{lemma:beta_nonsync}, if $\beta<-\frac{2}{3}$ and $n$ is divisible by three or $n$ is sufficiently large, then there exists a stable nonsynchronized stationary point. At this point, the Hessian has one zero eigenvalue which corresponds to translating each point by the same displacement and its other eigenvalues are negative. This implies that global synchronization does not occur.
\end{proof} 

\section{Generalized system}
\label{sec:generalizedsystem}

One of the extensions following Kuramoto's work was introduced by~\cite{Taylor2012} in the
following form:
\begin{equation}
\label{eq:generalsystem}
\dot{x}_i(t) = -\sum_{j=1}^n a_{i,j}f(x_i(t)-x_j(t)), \forall i\in[n],
\end{equation}
where $A=(a_{i,j})_{1\leq i,j\leq n}\in\mathbb{R}^{n\times n}$ is a weight matrix and
$f:\mathbb{T} \to \mathbb{R}$ is an interaction function.  Taylor~\cite{Taylor2012} showed
synchronization result for $f(x) = \sin(x)$ and $A$ being an adjacency matrix of an
(undirected) graph with each vertex having degree $\ge 0.94n$. Subsequent works eventually improved the lower bound
on degree to $0.75n$~\cite{dense_kuramoto}, while also showing graphs with each vertex of degree $\ge 0.6838n$, which do
not synchronize~\cite{lowerbound}. It is conjectured that there exists graphs with min-degree approaching $0.75n$ which do not synchronize. Furthermore, expander graphs have been utilized to
show that generating $A$ using a random process of adding edges leads to global synchronization
once $A$ is connected \cite{expander,randomgraph}.

In this section, we will extend our criterion to the special case of rank-1 matrices $A$. Initially,
we will only consider the following version:
\begin{equation}\tag{S3}
\label{eq:generalsystem3}
\dot{x}_i(t) = -\sum_{j=1}^n c_jf(x_i(t)-x_j(t)), \forall i\in[n],
\end{equation}
where $c_j> 0$ for $1\leq j\leq n$. This system generalizes \pref{eq:generalsystem2} by allowing different weights for each particle. We now state the
analogous stationarity and cut-stability conditions.

Suppose ${\bf x} = (x_i)_{1\leq i\leq n} \in \torus^n$. For the system~\eqref{eq:generalsystem3} a point is \textit{stationary} iff
\begin{equation}\tag{C3}
\label{eq:gencondition1}
\sum_{j=1}^n c_jf(x_i-x_j)=0, \, \forall i\in[n].
\end{equation}
We say that point $\bf x$ is \textit{cut-stable} if
\begin{equation}\tag{C4}
\label{eq:gencondition2}
\sum_{i\in S,\,j\in S^C} c_jf'(x_i - x_j) \geq 0
\end{equation} 
for all $S\subset [n]$ such that the value of $x_i$ is the same for all $i\in S$.

We state the following result, which generalizes \Cref{cor:main}. Note that \Cref{thm:main} is an implication of this corollary. 

\begin{theorem}
\label{thm:maingen}
Assume that $M$ is a positive real number such that for all stable, stationary, and non-synchronized points $\bf x$ of \pref{eq:generalsystem3}, $\tau_{\text{max}}({\bf x}) < M$. If 
\[
\tau\int_{-\pi}^\pi |f'''(x)|_+ dx\leq 4\left(1+\frac{\tau}{M}\right)f'(0),
\]
then every stationary and stable point $(x_1, \ldots, x_n)$ of system~\eqref{eq:generalsystem3} on $\mathbb{T}^n$ is synchronized, i.e. $x_1=\cdots=x_n$, where $\tau$ is as in \Cref{thm:main}.
\end{theorem}

\begin{proof}
The same proof of \Cref{thm:main} in \Cref{sec:mainproof} can be used, except with $\Psi(x)\triangleq\sum_{j=1}^n c_jf'(x-x_j)$ and $W_i\triangleq \sum_{j\in [n]: x_j=\theta_i} c_j$ replacing $N_i$. Similarly, we only require \pref{eq:gencondition2} for the proof.
\end{proof}

It is not immediately clear that global synchronization occurs in this setting, since we no longer have an obvious gradient ascent structure. First, we normalize \pref{eq:generalsystem3}. 

Assume that $g: \mathbb{T}^n\rightarrow \mathbb{R}_{>0}^n$ is smooth. Then, we can normalize the system as
\begin{equation}\tag{S4}
\label{eq:systemnormalized3}
\dot{x}_i(t) = -\frac{1}{g_i(x_1(t),\ldots,x_n(t))} \sum_{j=1}^n c_jf(x_i(t)-x_j(t)), \, 1\leq i\leq n,
\end{equation}
which allows us to state the following result, which generalizes \Cref{thm:ascentnormal} and is stated in the format of \Cref{cor:main}.

\begin{theorem}
\label{thm:ascentnormalgen}
Assume that $f(x)=\sin(x)h(\cos(x))$, where $h$ is a real-analytic function on an open set containing $[-1,1]$. Assume that $M$ is a positive real number such that for all stable, stationary, and non-synchronized points $\bf x$ of \pref{eq:systemnormalized3}, $\tau_{\text{max}}({\bf x}) < M$. Furthermore, assume that $\tau\product{1, |f'''|_+}\leq 4\left(1+\frac{\tau}{M}\right)f'(0)$, where $\tau$ is as in Theorem~\ref{thm:main}. Then, global synchronization occurs in \pref{eq:systemnormalized}.
\end{theorem}

Similarly to the approach of \Cref{sec:normalized}, we express the dynamical system in terms of points on $(\mathbb{S}^1)^n$. For $i\in [n]$ and $t\geq 0$, we let $p_i(t)=(\cos(x_i(t)), \sin(x_i(t)))$ to obtain the equivalent dynamical system
\begin{equation}\tag{S4'}
\label{eq:systemnormalized4}
\dot{\p}_i(t) = \frac{1}{g_i(\p(t))} \sum_{j=1}^n c_jh(\product{\p_i(t), \p_j(t)}) P_{\p_i(t)^\perp}(\p_j(t))\, \forall i\in [n],
\end{equation}
where $\p(t)=(\p_1(t),\ldots,\p_n(t))\in (\mathbb{S}^1)^n$ for $t\geq 0$.

Suppose $\varphi(x)=\int_0^x h(x) dx$ and let
\[
E_w(x_1,\ldots,x_n) = \frac{1}{2}\sum_{i,j=1}^n c_ic_j\varphi(\product{x_i,x_j}).
\]
Note that this energy function is also considered in \cite{synchronization_nonlinear}. The idea is that the $c_i$ correspond to the weights of the particles. We have the following generalization of \Cref{lemma:riemannian}. The result can be proved using the same approach.

\begin{lemma}
\label{lemma:riemanniangen}
Suppose the Riemannian manifold $M$ over $(\mathbb{S}^1)^n$ has positive-definite inner product $\product{(a_1, \ldots, a_n), (b_1, \ldots, b_n)}_P = \sum_{i=1}^n \alpha_i(P)a_ib_i$ at $P=(p_1,\ldots,p_n) \in(\mathbb{S}^1)^n$. Then,
\[
(\nabla_M)_i E_w(P) = \frac{1}{\alpha_i(P)} \sum_{j=1}^n c_ic_jh(\product{p_i, p_j}) P_{p_i^\perp}(p_j)\,\forall i\in[n].
\]
\end{lemma}

\begin{proof}[Proof of \Cref{thm:ascentnormalgen}] The same proof as the proof of \Cref{thm:ascentnormal} can be used. The only differences are as follows. The Riemannian manifold $R$ over $(\mathbb{S}^1)^n$ has positive-definite inner product $\product{(a_1,\ldots,a_n),(b_1,\ldots,b_n)}_P = \sum_{i=1}^n c_ig_i(P)a_ib_i$ at $P$. Of course, \pref{eq:systemnormalized4} replaces \pref{eq:systemnormalized2}, and we implement the remaining analogous replacements; for example, we replace \pref{eq:condition1} and \pref{eq:condition2} with \pref{eq:gencondition1} and \pref{eq:gencondition2}, respectively, as well as \Cref{thm:main} with \Cref{thm:maingen}. 

When verifying that \pref{eq:gencondition2} is true, the final expression we obtain is that for $S\subset [n]$ such that the $x_i$ are all equal to $\theta$ for $i\in S$, 
\[
\sum_{i\in S, j\notin S} c_ic_j f'(x_i-x_j) \geq 0 \Leftrightarrow \left(\sum_{i\in S} c_i\right)\sum_{j\notin S} c_j f'(\theta-x_j) \geq 0.
\]
Note that \pref{eq:gencondition2} is clearly true when $S$ is empty. If $S$ is nonempty, since the $c_i$ are positive we have that $\sum_{j\notin S} c_jf'(\theta-x_j) \geq 0$, so \pref{eq:gencondition2} holds.
\end{proof}

An important implication of \Cref{thm:ascentnormalgen} is the following result, which allows the setting of $A$ as $w_1w_2^{\top}$ in \pref{eq:generalsystem} while still having global synchronization.

\begin{corollary}
\label{cor:generalweights}
Assume that $f$ satisfies the conditions of \Cref{thm:ascentnormalgen}, where \pref{eq:systemnormalized3} is replaced by the system
\[
\dot{x}_i(t) = -\sum_{j=1}^n w_{1i}w_{2j}f(x_i(t)-x_j(t)), \, 1\leq i\leq n,
\]
where $w_{1i},w_{2i}>0$ for $1\leq i\leq n$ are fixed. Then, global synchronization occurs in this system.
\end{corollary}

\begin{proof}
This follows from \Cref{thm:ascentnormalgen} with $g_i=w_{1i}^{-1}$ and $c_i=w_{2i}$ for $1\leq i\leq n$.
\end{proof}

\bibliography{references.bib}

\appendix

\section{Results for \texorpdfstring{$\beta > \frac{1}{3}$}{}}

In this subsection, $f(x)=\sin(x)e^{\beta(\cos(x)-1)}$.

\begin{lemma}
\label{lemma:positiveregion}
Suppose $\beta> \frac{1}{3}$. There exists $0<b<a<\pi$ such that the positive region of $f'''(x)$ is $(-a, -b) \sqcup (b, a)$.
\end{lemma}

\begin{proof}
Let
\[
p(z) = -(z + 3\beta z^2 - 4\beta(1-z^2) - 6\beta^2 z(1-z^2) + \beta^3 (1-z^2)^2).
\]
Observe that
\[
f'''(x) = p(\cos(x))e^{\beta(\cos(x)-1)},
\]
so it suffices to analyze the positive regions of $p$ over $[-1, 1]$. First, observe that
\begin{align*}
& p(1) = -1 - 3\beta < 0,\, p(1-\frac{0.3}{\beta}) = 1.688 - 0.1761\beta^{-1} + 0.24\beta> 0, \\
& p(-1) = 1 - 3\beta < 0, \, p(-1-\frac{2}{\beta}) = 5\beta + 6\beta^{-1} + 13>0.
\end{align*}
Since $p(-1) < 0$ and $p(-1-\frac{2}{\beta}) > 0$, $p$ has two roots less than $-1$. Furthermore, since $p(1)<0$, $p(1-\frac{0.3}{\beta})>0$, and $p(-1)<0$, $p$ has two roots in $[-1, 1]$, and $p$ is positive between these two roots. This finishes the proof.
\end{proof}

\begin{lemma}
\label{lemma:tau1}
$\sqrt{x}\arccos(\frac{\sqrt{1+4x^2}-1}{2x})$ is strictly increasing over $(0, \infty)$.
\end{lemma}

\begin{proof}
We compute that
\[
\frac{d}{dx} \sqrt{x}\arccos(\frac{\sqrt{1+4x^2}-1}{2x}) = \frac{\arccos(\frac{\sqrt{1+4x^2}-1}{2x})-\frac{\sqrt{\sqrt{1+4x^2}-1}}{\sqrt{2x^2+\frac{1}{2}}}}{2\sqrt{x}}.
\]
Therefore, it suffices to prove that 
\[
\arccos(\frac{\sqrt{1+4x^2}-1}{2x}) >\frac{\sqrt{\sqrt{1+4x^2}-1}}{\sqrt{2x^2+\frac{1}{2}}} \Leftrightarrow\frac{\sqrt{1+4x^2}-1}{2x} < \cos(\frac{\sqrt{\sqrt{1+4x^2}-1}}{\sqrt{2x^2+\frac{1}{2}}}) .
\]
Observe that using $\cos(z)\geq 1-\frac{z^2}{2}$ over $(0, \pi)$ yields
\[
\cos(\frac{\sqrt{\sqrt{1+4x^2}-1}}{\sqrt{2x^2+\frac{1}{2}}}) \geq 1 - \frac{\sqrt{1+4x^2}-1}{1 + 4x^2}.
\]
Then, it suffices to prove that 
\begin{align*}
& \frac{\sqrt{1+4x^2}-1}{2x} < 1 - \frac{\sqrt{1+4x^2}-1}{1 + 4x^2}
\\ \Leftrightarrow & \sqrt{1+4x^2}(\frac{1}{2x} + \frac{1}{1+4x^2}) < 1 + \frac{1}{1+4x^2} + \frac{1}{2x} \\
\Leftrightarrow & \sqrt{1+4x^2}(1+2x+4x^2) < (1+4x^2)2x + 1+2x+4x^2 \\
\Leftrightarrow & (\sqrt{1+4x^2} - 1)(1+2x+4x^2) < (1+4x^2)2x \\
\Leftrightarrow & 2x(1+2x+4x^2) < (1+\sqrt{1+4x^2})(1+4x^2),
\end{align*}
which is straightforward to verify.
\end{proof}

The following corollary also appears in \cite[Lemma 4]{clustering_causal}.

\begin{corollary}
\label{corollary:tau2} $\arccos(\frac{\sqrt{1+4x^2}-1}{2x}) < \frac{1}{\sqrt{x}}$ over $(0, \infty)$.
\end{corollary}

\begin{proof}
This follows from \Cref{lemma:tau1} and $\lim_{x\rightarrow\infty} \sqrt{x}\arccos(\frac{\sqrt{1+4x^2}-1}{2x}) = 1$.
\end{proof}

\begin{lemma}
\label{lemma:beta1}
If $\beta \geq 1$ then $\tau(f''(a)-f''(b)) < 2$.
\end{lemma}

\begin{proof}
Let $w=1-\cos(x)$ and assume that $x\in [0, \pi]$ so that $\sin(x) \geq 0$. Then,
\begin{align*}
f''(x) & = \sqrt{2w-w^2}(-1 - 3\beta (1-w) + \beta^2 (2w-w^2))e^{-\beta w} \\
& = \sqrt{2w-w^2}(\beta(-3 + 2\beta w) + (-1+3\beta w - \beta^2 w^2))e^{-\beta w}.
\end{align*}
Let $z=\beta w$. Hence,
\[
f''(x) = \sqrt{\beta}\sqrt{z}\sqrt{2-\frac{z}{\beta}}(-3+2z + \frac{1}{\beta}(-1+3z-z^2))e^{-z}.
\]
Let 
\[
g_\beta(z) = \sqrt{z}\sqrt{2-\frac{z}{\beta}}(-3+2z + \frac{1}{\beta}(-1+3z-z^2))e^{-z}
\]
where $z\in [0, 2\beta]$. Note that $f''(x)=\sqrt{\beta}g_\beta(\beta-\beta\cos(x))$ so
\[
f'''(x)=\beta\sqrt{\beta}\sin(x)g'_\beta(\beta-\beta\cos(x))
\]
Since $x\in [0, \pi]$ and the positive region of $f'''$ in $[0, \pi]$ is $(b, a)$, the positive region of $g_\beta'$ is $(\beta(1-\cos(b)), \beta(1-\cos(a)))$. Particularly, $\frac{f''(a)-f''(b)}{\sqrt{\beta}} = g_\beta(\beta(1-\cos(a))) - g_\beta(\beta(1-\cos(b)))$. 

First, observe that $g'_\beta(0.18)< 0$ and $g'_\beta(1.4)>0$ for all $\beta \geq 1$. Therefore, because the positive region is continuous, we always have that 
\[
0.18 < \beta(1-\cos(b)) < 1.4 < \beta(1-\cos(a)).
\]
Moreover, $g_\beta(\beta(1-\cos(a))) > 0 > g_\beta(\beta(1-\cos(b)))$; this is because $g_\beta(0)=0$, $g_\beta(0.1)<0$, and $g_\beta(1.9)>0$, so $g_\beta$ first decreases to a negative value and then increases to a positive value.

Let
\[
g^1(z) = \sqrt{2z}(-3+2z)e^{-z},\, g^2(z) = \sqrt{2z}(-1+3z-z^2)e^{-z}.
\]
Then, $g_\beta$ is similar to a linear combination of $g^1$ and $g^2$, with
\[
g_\beta(z) = \frac{\sqrt{2-\frac{z}{\beta}}}{\sqrt{2}} (g^1(z) + \frac{1}{\beta}g^2(z)).
\]
For the following computations, we utilize properties of $g^1$ and $g^2$ which are straightforward to verify.

Because $0.18<\beta(1-\cos(b))<1.4$, the value of $\sqrt{2-\frac{z}{\beta}}g^2(z)$ at $z=\beta(1-\cos(b))$ is at least 
\[
\sqrt{2-\frac{0.18}{\beta}}g^2(0.18)  > -0.247\sqrt{2-\frac{0.18}{\beta}}.
\]
Furthermore, the value of $\sqrt{2-\frac{z}{\beta}}g^1(z)$ at $z=\beta(1-\cos(b))$ is at least
\[
-1.381\sqrt{2-\frac{0.18}{\beta}},
\]
since $g^1 > -1.381$. Therefore, the value of $g_\beta$ at $z=\beta(1-\cos(b))$ is greater than
\begin{equation}
\label{eq:minval}
-0.247\frac{\sqrt{2-\frac{0.18}{\beta}}}{\sqrt{2}\beta} - 1.381 \frac{\sqrt{2-\frac{0.18}{\beta}}}{\sqrt{2}}.
\end{equation}

Since $\beta(1-\cos(a)) > 1.4$, the value of $\sqrt{2-\frac{z}{\beta}}g^2(z)$ at $z=\beta(1-\cos(a))$ is at most
\[
\sqrt{2-\frac{1.4}{\beta}}g^2(1.4) < 0.512 \sqrt{2-\frac{1.4}{\beta}}.
\]
Furthermore, the value of $\sqrt{2-\frac{z}{\beta}}g^1(z)$ at $z=\beta(1-\cos(a))$ is at most 
\[
0.375 \sqrt{2-\frac{1.5}{\beta}},
\]
since $g^1\leq 0$ for $z\leq 1.5$ and $g^1 < 0.375$. Therefore, the value of $g_\beta$ at $z=\beta(1-\cos(a))$ is less than
\[
0.512 \frac{\sqrt{2-\frac{1.4}{\beta}}}{\sqrt{2}\beta} + 0.375 \frac{\sqrt{2-\frac{1.5}{\beta}}}{\sqrt{2}}.
\]
Using this inequality and \pref{eq:minval} gives that
\begin{align*}
& g_\beta(\beta(1-\cos(a))) - g_\beta(\beta(1-\cos(b))) <\\ & 0.512 \frac{\sqrt{2-\frac{1.4}{\beta}}}{\sqrt{2}\beta} + 0.375 \frac{\sqrt{2-\frac{1.5}{\beta}}}{\sqrt{2}} + 0.247\frac{\sqrt{2-\frac{0.18}{\beta}}}{\sqrt{2}\beta} + 1.381 \frac{\sqrt{2-\frac{0.18}{\beta}}}{\sqrt{2}}.
\end{align*}
Let
\[
\varphi(\beta) = 0.512 \frac{\sqrt{2-\frac{1.4}{\beta}}}{\sqrt{2}\beta} + 0.375 \frac{\sqrt{2-\frac{1.5}{\beta}}}{\sqrt{2}} +0.247\frac{\sqrt{2-\frac{0.18}{\beta}}}{\sqrt{2}\beta} + 1.381 \frac{\sqrt{2-\frac{0.18}{\beta}}}{\sqrt{2}}
\]
for $\beta\geq 1$, so that $g_\beta(\beta(1-\cos(a))) - g_\beta(\beta(1-\cos(b)))  < \varphi(\beta)$.

If $\beta\geq 2$ then $\varphi(\beta) < 2$, so 
\[
\tau(f''(a) - f''(b)) = \sqrt{\beta}\tau(g_\beta(\beta(1-\cos(a))) - g_\beta(\beta(1-\cos(b)))) < \sqrt{\beta}\tau \varphi(\beta) < 2\sqrt{\beta}\tau.
\]
Since $\sqrt{\beta}\tau<1$ by \Cref{corollary:tau2}, $\tau (f''(a) - f''(b)) < 2$.

Assume that $\beta \in [1, 2)$. Then, from \Cref{lemma:tau1}, $\sqrt{\beta}\tau < \sqrt{2}\tau(2) < 0.96$, so $\tau (f''(a) - f''(b)) < 0.96 \varphi(\beta) < 2$.
\end{proof}

\begin{lemma}
\label{lemma:beta_0.75}
Suppose $\beta\in [0.75, 1)$. Then, $\tau(f''(a)-f''(b))<2$.
\end{lemma}

\begin{proof}
Observe that $g'_\beta(0.165) < 0$ and $g'_\beta(1.24) > 0$ for $\beta\in [0.75, 1)$. Thus,
\[
0.165 < \beta(1-\cos(b)) < 1.24 < \beta(1-\cos(a)).
\]
Similarly, $g_\beta(\beta(1-\cos(a))) > 0 > g_\beta(\beta(1-\cos(b)))$, because $g_\beta(0)=0$, $g_\beta(0.1)<0$, $g_\beta(1.4)>0$. Therefore,
\begin{align*}
& g_\beta(\beta(1-\cos(a))) - g_\beta(\beta(1-\cos(b))) <\\ & 0.54 \frac{\sqrt{2-\frac{1.24}{\beta}}}{\sqrt{2}\beta} + 0.375 \frac{\sqrt{2-\frac{1.5}{\beta}}}{\sqrt{2}} + 0.26\frac{\sqrt{2-\frac{0.165}{\beta}}}{\sqrt{2}\beta} + 1.381 \frac{\sqrt{2-\frac{0.165}{\beta}}}{\sqrt{2}} =: \varphi(\beta)
\end{align*}
over $[0.75, 1)$. Using \Cref{lemma:tau1}, $\sqrt{\beta}\tau < \tau(1) < 0.91$, so $\tau(f''(a)-f''(b)) < 0.91\varphi(\beta) < 2$.
\end{proof}

\begin{lemma}
\label{lemma:beta_0.5}
Suppose $\beta\in [0.5, 0.75)$. Then, $\tau(f''(a)-f''(b)) < 2$.
\end{lemma}

\begin{proof}
Observe that $g'_\beta(0.14) < 0$ and $g'_\beta(0.9) > 0$ for $\beta\in [0.5, 0.75)$. Thus,
\[
0.14 < \beta(1-\cos(b)) < 0.9 < \beta(1-\cos(a)).
\]
Similarly, $g_\beta(\beta(1-\cos(a))) > 0 > g_\beta(\beta(1-\cos(b)))$, because $g_\beta(0)=0$, $g_\beta(0.1)<0$, $g_\beta(0.99)>0$. Therefore,
\begin{align*}
& g_\beta(\beta(1-\cos(a))) - g_\beta(\beta(1-\cos(b))) <\\ & 0.542 \frac{\sqrt{2-\frac{0.9}{\beta}}}{\sqrt{2}\beta} + 0.28\frac{\sqrt{2-\frac{0.14}{\beta}}}{\sqrt{2}\beta} + 1.381 \frac{\sqrt{2-\frac{0.14}{\beta}}}{\sqrt{2}} =: \varphi(\beta)
\end{align*}
over $[0.5, 0.75)$; observe that we have removed the term $0.375 \frac{\sqrt{2-\frac{1.5}{\beta}}}{\sqrt{2}}$ for $g^1$, since $g^1$ is always non-positive when $\beta \leq 0.75$ and $z\leq 2\beta$. Using \Cref{lemma:tau1}, $\sqrt{\beta}\tau < \sqrt{0.75}\tau(\sqrt{0.75}) < 0.88$, so $\tau(f''(a)-f''(b)) < 0.88\varphi(\beta) < 2$.
\end{proof}

\begin{lemma}
\label{lemma:beta_0.33}
Suppose $\beta\in (\frac{1}{3}, 0.5)$. Then, $\tau(f''(a)-f''(b)) < 2$.
\end{lemma}

\begin{proof}
Observe that $g'_\beta(0.121) < 0$ and $g'_\beta(\frac{2}{3}) > 0$ for $\beta\in (\frac{1}{3}, 0.5)$. Thus,
\[
0.12 < \beta(1-\cos(b)) < \frac{2}{3} < \beta(1-\cos(a)).
\]
Similarly, $g_\beta(\beta(1-\cos(a))) > 0 > g_\beta(\beta(1-\cos(b)))$, because $g_\beta(0)=0$, $g_\beta(0.1)<0$, $g_\beta(\frac{2}{3})>0$.

Since $z\leq 2\beta < 1$, the maximal positive value of $g_2$ is less than $g_2(1) < 0.521$. Therefore,
\begin{align*}
& g_\beta(\beta(1-\cos(a))) - g_\beta(\beta(1-\cos(b))) <\\ & 0.521 \frac{\sqrt{2-\frac{2/3}{\beta}}}{\sqrt{2}\beta} + 0.285\frac{\sqrt{2-\frac{0.121}{\beta}}}{\sqrt{2}\beta} + 1.381 \frac{\sqrt{2-\frac{0.121}{\beta}}}{\sqrt{2}} =: \varphi(\beta)
\end{align*}
over $(\frac{1}{3}, 0.5)$; similarly, we have removed the positive term for $g^1$. Using \Cref{lemma:tau1}, $\sqrt{\beta}\tau < \sqrt{0.5}\tau(0.5) < 0.809$, so $\tau(f''(a)-f''(b)) < 0.809\varphi(\beta) < 2$.
\end{proof}

\section{Results for \texorpdfstring{$0<\beta \leq \frac{1}{3}$}{}}

In this subsection, $f(x)=\sin(x)e^{\beta(\cos(x)-1)}$.

\begin{lemma}
\label{lemma:positiveregion2}
Suppose $\beta\in (0, \frac{1}{3}]$. Then, there exists $a\in (0, \frac{\pi}{2})$ such that the nonnegative region of $f'''(x)$ over $[0, 2\pi)$ is $[a, 2\pi - a]$.
\end{lemma}

\begin{proof}
We use the same method as the proof of \Cref{lemma:positiveregion}. Observe that
\begin{align*}
& p(1) = -1 - 3\beta < 0,\, p(0) = -\beta(\beta^2-4) > 0, \, p(-1) = 1 - 3\beta \geq 0, \\
& p(-1-\frac{1}{\beta}) = 5\beta - \beta^{-1} + 1 < 0, \,p(-1-\frac{2}{\beta}) = 5\beta + 6\beta^{-1} + 13>0,
\end{align*}
and $\lim_{z\rightarrow-\infty} p(z) = -\infty$. Thus, $p$ has a root in each of the following intervals: $(0, 1)$, $(-1-\frac{1}{\beta}, -1]$, $(-1 - \frac{2}{\beta}, -1 - \frac{1}{\beta})$, and $(-\infty, -1-\frac{2}{\beta})$. 

Let $r$ be the root of $p$ in $(0, 1)$. Let $a=\arccos(r)$. Because $r\in (0, 1)$, $a\in (0, \frac{\pi}{2})$. Furthermore, because $p(1) < 0$, $p(0) > 0$, and $p$ has no other roots in $(-1, 1)$, $p(z)$ is nonnegative for $z\in [-1, 1]$ if and only if $z \in [-1, r]$. Therefore, the nonnegative region of $f'''$ is $[a, 2\pi - a]$.
\end{proof}

\begin{remark}
In contrast with \Cref{lemma:positiveregion}, we consider the nonnegative region of $f'''$ rather than the positive region. The reason for this is that when $\beta=\frac{1}{3}$, $p(-1)=0$, so the positive region would be $(a, \pi)\cup (\pi, 2\pi-a)$ for this case. For simplicity, we consider the nonnegative region.
\end{remark}

\begin{corollary}
Suppose $\beta\in (0, \frac{1}{3}]$. Then, $f'' \leq 0$ over $[0, \pi]$ and $f''\geq 0$ over $[\pi, 2\pi]$.
\end{corollary}

\begin{proof}
Using \Cref{lemma:positiveregion2}, assume that the nonnegative region of $f'''$ is $[a, 2\pi-a]$ for $a\in (0, \frac{\pi}{2})$. Then, since $f''(0)=f''(\pi) = 0$, we have that over $[0, \pi]$, $f''$ first decreases from $0$ to its minimal value at $a$ and then increases to $0$, so $f''$ is non-positive. Because $f''$ is odd, $f''$ is nonnegative over $[\pi, 2\pi]$.
\end{proof}

\begin{lemma}
\label{lemma:beta_final}
Suppose $\beta\in (0, \frac{1}{3}]$. Then, $\tau f''(a)> -2$.
\end{lemma}

\begin{proof}
We have that 
\[
f''(a) \geq -(1+3\beta) \sin(a)e^{\beta\cos(a)}e^{-\beta} \geq -(1+3\beta)\sin(\tau)e^{\beta\cos(\tau)}e^{-\beta}.
\]
Thus, it suffices to prove that
\[
(1+3\beta)\sin(\tau)e^{\beta\cos(\tau)}e^{-\beta}\tau < 2.
\]
Observe that 
\[
0\leq \cos(\tau) = \beta\sin(\tau)^2 \leq \beta,
\]
so it suffices to prove that
\[
(1+3\beta)e^{\beta^2-\beta} \sin(\tau)\tau < 2.
\]

Assume that $0<\beta \leq 0.148$. Then, $\tau \leq \frac{\pi}{2}$ and $\sin(\tau) \leq 1$. Because
\[
(1+3\beta)e^{\beta^2 - \beta}\cdot \frac{\pi}{2} < 2,
\]
we have that $\tau f''(a) > -2$. ($0.148$ is approximately the maximal value of $\beta$ for which this method is correct.)

Assume that $0.148 < \beta \leq 0.228$. Then, $\tau < \tau(0.148) < 1.43$ and $\sin(\tau) < \sin(1.43)$. Because
\[
(1+3\beta)e^{\beta^2-\beta}\cdot 1.43\sin(1.43) < 2,
\]
we have that $\tau f''(a) > -2$.

Assume that $0.228 < \beta \leq 0.278$. Then, $\tau < \tau(0.228) < 1.36$ and $\sin(\tau) < \sin(1.36)$. Because
\[
(1+3\beta)e^{\beta^2 - \beta}\cdot 1.36\sin(1.36) < 2,
\]
we have that $\tau f''(a) > -2$.

Assume that $0.278 < \beta \leq 0.321$. Then, $\tau < \tau(0.278) < 1.31$ and $\sin(\tau) < \sin(1.31)$. Because 
\[
(1+3\beta)e^{\beta^2-\beta}\cdot 1.31\sin(1.31) < 2,
\]
we have that $\tau f''(a) > -2$.

Assume that $0.321< \beta \leq \frac{1}{3}$. Then, $\tau < \tau(0.321) < 1.28$ and $\sin(\tau) < \sin(1.28)$. Because 
\[
(1+3\beta)e^{\beta^2-\beta}\cdot 1.28\sin(1.28) < 2,
\]
we have that $\tau f''(a) > -2$.
\end{proof}

\section{Results for \texorpdfstring{$-\frac{1}{3}<\beta<0$}{}}

In this subsection, $f(x)=\sin(x)e^{\beta(\cos(x)-1)}$.

\begin{lemma}
\label{lemma:positiveregion3}
Suppose $\beta\in (-\frac{1}{3}, 0)$. Then, there exists $a\in (\frac{\pi}{2}, \pi)$ such that the nonnegative region of $f'''(x)$ over $[0, 2\pi)$ is $[a, 2\pi - a]$.
\end{lemma}

\begin{proof}
We use the same method as the proof of \Cref{lemma:positiveregion}. Observe that
\begin{align*}
& p(1) = -1 - 3\beta < 0,\, p(0) = -\beta(\beta^2-4) < 0, \, p(-1) = 1 - 3\beta > 0, \\
& p(1-\frac{1}{\beta}) = 5\beta - \beta^{-1} - 1 > 0, \,p(1-\frac{2}{\beta}) = 5\beta + 6\beta^{-1} - 13 < 0,
\end{align*}
and $\lim_{z\rightarrow\infty} p(z) =\infty$. Thus, $p$ has a root in each of the following intervals: $(-1, 0)$, $(1, 1-\frac{1}{\beta})$, $(1 - \frac{1}{\beta}, 1 - \frac{2}{\beta})$, and $(1-\frac{2}{\beta}, \infty)$. 

Let $r$ be the root of $p$ in $(-1, 0)$. Let $a=\arccos(r)$. Because $r\in (-1, 0)$, $a\in (\frac{\pi}{2}, \pi)$. Furthermore, because $p(-1) > 0$, $p(0) < 0$, and $p$ has no other roots in $(-1, 1)$, $p(z)$ is nonnegative for $z\in [-1, 1]$ if and only if $z \in [-1, r]$. Therefore, the nonnegative region of $f'''$ is $[a, 2\pi - a]$.
\end{proof}

\begin{remark}
If $\beta=-\frac{1}{3}$, the nonnegative region of $f'''(x)$ over $[0, 2\pi)$ is $[a, 2\pi-a]\cup \{0\}$, since $p(1)=0$.
\end{remark}

\begin{corollary}
Suppose $\beta\in (-\frac{1}{3}, 0)$. Then, $f'' \leq 0$ over $[0, \pi]$ and $f''\geq 0$ over $[\pi, 2\pi]$.
\end{corollary}

\begin{proof}
Using \Cref{lemma:positiveregion2}, assume that the nonnegative region of $f'''$ is $[a, 2\pi-a]$ for $a\in (\frac{\pi}{2}, \pi)$. Then, since $f''(0)=f''(\pi) = 0$, we have that over $[0, \pi]$, $f''$ first decreases from to $0$ to its minimal value at $a$ and then increases to $0$, so $f''$ is non-positive. Because $f''$ is odd, $f''$ is nonnegative over $[\pi, 2\pi]$.
\end{proof}

\begin{lemma}
\label{lemma:semicircle}
Suppose $\beta<0$. If $\mathbf{x}$ is a stable, stationary, and non-synchronized point of \pref{eq:generalsystem2}, then $\tau_{\text{max}}(\mathbf{x})<\pi$.
\end{lemma}

\begin{proof}
This is implied by \cite[Lemma 10]{synchronization_nonlinear} with $\varphi(t)=-e^{\beta t}$.
\end{proof}

\begin{lemma}
\label{lemma:beta_negative}
Suppose $\beta\in [-0.16, 0)$. Then, $\tau f''(a) \geq -2\left(1+\frac{\tau}{\pi}\right)$.
\end{lemma}

\begin{proof}
We have that 
\[
f''(a) \geq -(1-3\beta) \sin(a)e^{\beta\cos(a)}e^{-\beta} \geq -(1-3\beta)\sin(\tau)e^{\beta\cos(\tau)}e^{-\beta}.
\]
Thus, it suffices to prove that
\[
(1-3\beta)\sin(\tau)e^{\beta\cos(\tau)}e^{-\beta}\tau \leq 2\left(1+\frac{\tau}{\pi}\right).
\]
Observe that 
\[
0\geq \cos(\tau) = \beta\sin(\tau)^2 \geq \beta,
\]
so it suffices to prove that
\[
(1-3\beta)e^{\beta^2-\beta} \sin(\tau)\tau \leq 2\left(1+\frac{\tau}{\pi}\right).
\]
Equivalently, it suffices to prove that
\[
\tau\left((1-3\beta)e^{\beta^2-\beta} \sin(\tau) - \frac{2}{\pi}\right) \leq 2.
\]

Assume that $\beta \in [-0.16, 0)$. Observe that both $\tau$ and $(1-3\beta)e^{\beta^2-\beta} - \frac{2}{\pi}$ increase as $\beta$ decreases from $0$. At $\beta=-0.16$, we have that
\[
\tau\left((1-3\beta)e^{\beta^2-\beta} - \frac{2}{\pi}\right) \leq 2,
\]
so this must be the case for all $\beta\in [-0.16, 0)$. Thus,
\[
\tau\left((1-3\beta)e^{\beta^2-\beta} \sin(\tau) - \frac{2}{\pi}\right) \leq \tau\left((1-3\beta)e^{\beta^2-\beta} - \frac{2}{\pi}\right) \leq 2
\]
for all $\beta\in [-0.16, 0)$.
\end{proof}

The lower bound $-0.16$ for $\beta$ in \Cref{lemma:beta_negative} can be improved, but we omit this refinement for simplicity. We establish that global synchronization does not occur for all negative values of $\beta$.

\begin{lemma}
\label{lemma:beta_nonsync}
Suppose $\beta < -\frac{2}{3}$. Then, if $n$ is a multiple of $3$ or sufficiently large, there exists a value of $\mathbf{x}\in\mathbb{T}^n$ such that:
\begin{enumerate}
\item There exists three elements $p_1$,$p_2$, and $p_3$ of $\mathbb{T}$ such that $\lfloor\frac{n}{3}\rfloor$, $\lfloor\frac{n}{3}\rfloor$, and $n-2\lfloor\frac{n}{3}\rfloor$ points of $\mathbf{x}$ are placed at $p_1$, $p_2$, and $p_3$, respectively.
\item The vector $\mathbf{x}$ is a critical point of $E$ and the Hessian of $E$ over $\mathbb{T}^n$ at $\mathbf{x}$ is negative semidefinite, with only one eigenvalue whose eigenvectors are the scalar multiples of the vector $[1,\ldots,1]^{\top}$.
\end{enumerate}
\end{lemma}

\begin{proof}
Suppose $n\geq 3$. Let $p_1=0$, $p_2=\alpha$, and $p_3=2\pi-\alpha$, where $\alpha$ is an element of $(0,\pi)$ such that 
\[
\left\lfloor\frac{n}{3}\right\rfloor\sin(2\alpha)e^{\beta\cos(2\alpha)}  + \left(n-2\left\lfloor\frac{n}{3}\right\rfloor\right)\sin(\alpha)e^{\beta\cos(\alpha)} = 0;
\]
as $n\rightarrow\infty$, we can set $\alpha=\frac{2\pi}{3}+o_n(1)$ and in particular we can set $\alpha=\frac{2\pi}{3}$ when $n$ is a multiple of $3$. 

The
Hessian of the energy function $E$ is the Laplacian $L$ of $[-f'(\mathbf{x}_i-\mathbf{x}_j)]_{i,j=1}^n$.  Recall that $f'(x)=(\cos(x)-\beta\sin(x)^2)e^{\beta(\cos(x)-1)}$ and the Laplacian of a symmetric $n\times n$ matrix $A$ is $D-A$, where $D$ is the diagonal matrix with diagonal $[\sum_{i=1}^n A_{ij}]_{1\leq j\leq n}$. Furthermore, for a vector $v$, 
\[
v^{\top}Lv = -\sum_{i,j=1}^n \frac{1}{2}f'(\mathbf{x}_i-\mathbf{x}_j)(v_i-v_j)^2.
\]
However, we always have that $|\mathbf{x}_i-\mathbf{x}_j| \in \{0, \alpha, 2\pi-\alpha\}$. Since $\beta<-\frac{2}{3}$ and $\alpha=\frac{2\pi}{3}+o_n(1)$, $f'(x)>0$ for all $x\in\{0,\alpha,2\pi-\alpha\}$ when $n$ is sufficiently large; if $n$ is a multiple of $3$, we can set $\alpha=\frac{2\pi}{3}$ to obtain that $f'(x)>0$ for all $x\in\{0, \alpha, 2\pi-\alpha\}$. Thus, condition 2 is satisfied when $n$ is a multiple of $3$ or sufficiently large.
\end{proof}

\end{document}